\documentclass[reqno,a4paper]{amsart}
\usepackage[utf8]{inputenc}
\usepackage[T1]{fontenc}
\usepackage[british]{babel}
\usepackage{amsfonts, amssymb, amsmath}
\usepackage[all]{xy}
\usepackage{hyperref}
\usepackage{mathbbol}
\usepackage{calrsfs}
\usepackage{mathabx}
\usepackage{amsfonts}
\usepackage{amssymb}
\usepackage{amsthm}
\usepackage{amsmath}
\allowdisplaybreaks
\usepackage{xcolor}
\definecolor{l-gray}{gray}{0.7}

\theoremstyle{plain}
\newtheorem{theorem}{Theorem}[subsection]
\newtheorem{proposition}[theorem]{Proposition}
\newtheorem{corollary}[theorem]{Corollary}
\newtheorem{lemma}[theorem]{Lemma}

\newtheorem*{thm}{Theorem}

\theoremstyle{definition}
\newtheorem{definition}[theorem]{Definition}

\theoremstyle{remark}
\newtheorem{remark}[theorem]{Remark}
\newtheorem{remarks}[theorem]{Remarks}
\newtheorem{example}[theorem]{Example}

\numberwithin{equation}{section}
\numberwithin{figure}{section}

\newenvironment{tfae}
{
\begin{enumerate}}
{\end{enumerate}}

\newcommand{\comp}{}
\newcommand{\defeq}{\coloneq}

\newcommand{\noproof}{\hfill \qed}

\newcommand{\couniv}[2]{\langle #1, #2\rangle}
\newcommand{\univ}[2]{(#1, #2)}
\newcommand{\qeq}{\phantom{\quad}=\phantom{\quad}}
\newcommand{\equal}{\overset{}{\qeq}}

\newcommand{\defn}{\textbf}

\newcommand{\To}{\Rightarrow}

\newcommand{\del}{\partial}

\newcommand{\cupdot}{\mathbin{\mathaccent\cdot\cup}} %command for the disjoint union in Set

\DeclareMathOperator{\tw}{tw}
\DeclareMathOperator{\Ker}{Ker}

\newcommand{\h}{\mathbb{h}}

\newcommand{\C}{\ensuremath{\mathcal{C}}}

\newcommand{\E}{\ensuremath{\mathcal{E}}}

\newcommand{\V}{\mathcal{V}}

\newcommand{\Cat}{\ensuremath{\mathsf{Cat}}}

\newcommand{\Gp}{\ensuremath{\mathsf{Gp}}} %GRP
\newcommand{\Ab}{\ensuremath{\mathsf{Ab}}} %Abelian groups

\newcommand{\XMod}{\ensuremath{\mathsf{XMod}}}

\newcommand{\Ens}{\mathsf{Set}}
\newcommand{\Set}{\mathsf{Set}}
\newcommand{\LieK}{\mathsf{Lie_{\mathbb{K}}}}

\newcommand{\Ch}{\ensuremath{\mathsf{Ch}}}
\newcommand{\s}{\ensuremath{\mathsf{s}}}

\renewcommand{\P}{{\rm (P)}}

\newcommand{\A}{\mathbb{A}} 
\newcommand{\B}{\ensuremath{\mathbb{B}}}
\newcommand{\G}{\mathbb{G}}
\newcommand{\K}{\mathbb{K}}
\newcommand{\NN}{\mathbb{N}}

\renewcommand{\S}{\mathbb{S}}
\newcommand{\ZZ}{\mathbb{Z}}

\newcommand{\Left}{\mathrm{L}}

\newcommand{\ncat}{$n$-cat}
\newcommand{\CatGrp}[1]{\text{$#1$-$\mathsf{Cat}$-$\mathsf{Gp}$}}
\newcommand{\nCatGrp}{\CatGrp{n}}
\newcommand{\OrdGrp}{\ensuremath{\mathsf{OrdGp}}}

\DeclareMathOperator{\Coker}{Coker}

\DeclareMathOperator{\Homology}{H}
\renewcommand{\H}{\Homology}
\DeclareMathOperator{\Z}{Z}
\DeclareMathOperator{\N}{N}
\DeclareMathOperator{\Hom}{Hom}
\DeclareMathOperator{\Der}{Der}
\DeclareMathOperator{\Ima}{Im}

\DeclareMathOperator{\op}{op}

\newdir{>>}{{}*!/3.5pt/:(1,-.2)@^{>}*!/3.5pt/:(1,+.2)@_{>}*!/7pt/:(1,-.2)@^{>}*!/7pt/:(1,+.2)@_{>}} % cokernel map
\newdir{ >>}{{}*!/8pt/@{|}*!/3.5pt/:(1,-.2)@^{>}*!/3.5pt/:(1,+.2)@_{>}}
\newdir{ |>}{{}*!/-3.5pt/@{|}*!/-8pt/:(1,-.2)@^{>}*!/-8pt/:(1,+.2)@_{>}} % kernel map
\newdir{ >}{{}*!/-8pt/@{>}}
\newdir{>}{{}*:(1,-.2)@^{>}*:(1,+.2)@_{>}}
\newdir{<}{{}*:(1,+.2)@^{<}*:(1,-.2)@_{<}}

\def\pullback{% with thanks to Valerian Even
\ar@{-}[]+R+<6pt,-1pt>;[]+RD+<6pt,-6pt>%
\ar@{-}[]+D+<1pt,-6pt>;[]+RD+<6pt,-6pt>}

\def\reversepullback{ %command of Maxime Culot to have the pullback symbol bottom left
\ar@{-}[]+R+<6pt,-1pt>;[]+RU+<6pt,6pt>%
\ar@{-}[]+U+<-1pt,6pt>;[]+RU+<6pt,6pt>}

\def\dottedpullback{%
\ar@{.}[]+R+<6pt,-1pt>;[]+RD+<6pt,-6pt>%
\ar@{.}[]+D+<1pt,-6pt>;[]+RD+<6pt,-6pt>}

\def\halfsplitpullback{%
\ar@{-}[]+R+<6pt,-1pt>;[]+RD+<6pt,-6pt>%
\ar@{-}[]+D+<.52ex,-6pt>;[]+RD+<6pt,-6pt>}

\def\pushout{%
\ar@{-}[]+L+<-6pt,1pt>;[]+LU+<-6pt,6pt>%
\ar@{-}[]+U+<-1pt,6pt>;[]+LU+<-6pt,6pt>}

\usepackage{tikz-cd}

\usetikzlibrary{cd,decorations.markings}
\tikzset{
    triarrow/.style={ % name for the arrow "triarrow"
        ->,
        postaction={decorate,-},
        decoration={
            markings,
            mark=at position .1 with {
                \arrow{Triangle[fill=white, length=6pt,width=6pt]}
            }
        }
    }
}
\tikzset{
    bigtriarrow/.style={ % name for the arrow "bigtriarrow"
        ->,
        postaction={decorate,-},
        decoration={
            markings,
            mark=at position .05 with {
                \arrow{Triangle[fill=white, length=6pt,width=6pt]}
            }
        }
    }
}

\DeclareMathOperator{\kker}{ker}
\DeclareMathOperator{\coker}{coker}

\usepackage{hyperref}
\hypersetup{
    colorlinks=true,       % false: liens encadrés; true: liens colorés
    linkcolor=blue,          % couleur des liens (ou bordures) internes
    citecolor=magenta,        % couleur des liens (ou bordures) vers bilbio
    filecolor=green,      % couleur des liens (ou bordures) vers fichiers
    urlcolor=cyan           % couleur des liens (ou bordures) url
   }

\allowdisplaybreaks


\begin{document}

\title[Non-additive derived functors via chain resolutions]{Non-additive derived functors\\ via chain resolutions}

\author{Maxime Culot}
\author{Fara Renaud}
\author{Tim Van der Linden}

\email{maxime.culot@uclouvain.be}
\email{fara.j.renaud@gmail.com}
\email{tim.vanderlinden@uclouvain.be}

\address[Maxime Culot, Fara Renaud, Tim Van der Linden]{Institut de Recherche en Math\'ematique et Physique, Universit\'e catholique de Louvain, che\-min du cyclotron~2 bte~L7.01.02, B--1348 Louvain-la-Neuve, Belgium}

\address[Tim Van der Linden]{Mathematics \& Data Science, Vrije Universiteit Brussel, Pleinlaan 2, B--1050 Brussel, Belgium}

\thanks{As a Ph.D.\ student, the second author was funded by \emph{Formation à la Recherche dans l'Industrie et dans l'Agriculture (FRIA)}, part of the Fonds de la Recherche Scientifique--FNRS. The third author is a Senior Research Associate of the Fonds de la Recherche Scientifique--FNRS. Competing interests: The authors declare none}

\begin{abstract}
	Let $F\colon \C \to \E$ be a functor from a category $\C$ to a homological (Borceux--Bourn~\cite{Borceux-Bourn}) or semi-abelian (Janelidze--Márki--Tholen~\cite{Janelidze-Marki-Tholen}) category~$\E$. We investigate conditions under which the homology of an object $X$ in $\C$ with coefficients in the functor $F$, defined via projective resolutions in $\C$, remains independent of the chosen resolution. Consequently, the left derived functors of $F$ can be constructed analogously to the classical abelian case.

	Our approach extends the concept of \emph{chain homotopy} to a non-additive setting using the technique of \emph{imaginary morphisms}. Specifically, we utilize the \emph{approximate subtractions} of Bourn--Janelidze~\cite{DB-ZJ-2009}, originally introduced in the context of subtractive categories~\cite{ZJanelidze-Subtractive, Ursini3}. This method is applicable when $\C$ is a pointed regular category with finite coproducts and enough projectives, provided the class of projectives is \emph{closed under protosplit subobjects}, a new condition introduced in this article and naturally satisfied in the abelian context. We further assume that the functor $F$ meets certain exactness conditions: for instance, it may be protoadditive and preserve proper morphisms and binary coproducts---conditions that amount to additivity when $\C$ and $\E$ are abelian categories.

	Within this framework, we develop a basic theory of derived functors, compare it with the simplicial approach, and provide several examples.
\end{abstract}

\maketitle

\tableofcontents

\section{Introduction}

\subsection{Derived functors in abelian categories}\label{sec - derived functor ab case}
In classical homological algebra, such as for modules over a ring or within the broader context of abelian categories, we can define \emph{left derived functors} for any additive functor $F\colon \C \to \E$, where $\C$ and~$\E$ are abelian categories and $\C$ has enough projectives. This condition allows us to construct a projective resolution $C(X)$ for any $X \in \C$, as illustrated in Figure~\ref{Figure Projective Resolution Intro}.
\begin{figure}
	\resizebox{\textwidth}{!}{
	\xymatrix{
	\cdots\ar[r] & C_{n+1}\ar[rr]^-{d_{n+1}}\ar@{-{ >>}}[dr]_-{\bar{d}_{n+1}} & & C_n \ar[r]^-{d_{n}} & C_{n-1}\ar[r] & \cdots \ar[r] & C_1 \ar[r]^-{d_1} & C_0\ar@{-{ >>}}[rd]_-{\coker(d_1)} \ar[rr] && 0 \\
	&  & \Ker(d_n) \ar@{{ |>}->}[ur]_-{\ker(d_n)} & & & & & & X
	}}
	\caption{A projective resolution $C(X)$ of an object $X$ in a normal category: the horizontal sequence is exact in all $C_n$ for $n \geq 1$ and $\H_0(C) = X$}\label{Figure Projective Resolution Intro}
\end{figure}
For $n \in \ZZ$, we define
\[
	\Left_n(F)(X) \coloneq \H_n(F(C(X))).
\]
The functor $\Left_n(F)\colon \C \to \E$ is called the \defn{$n$th left derived functor of $F$}. Its value at $X$ is sometimes referred to as the \defn{$n$th homology of $X$ with coefficients in~$F$}.

In addition to encompassing well-known examples such as the Tor and Ext functors, this concept is powerful due to the relationships between the derived functors of different dimensions of a given functor, as expressed via a long exact homology sequence (see, for instance,~\cite{Cartan-Eilenberg,MacLane:Homology,Weibel}). The construction of this long exact sequence in homology---see, e.g., the proof in~\cite{Cartan-Eilenberg}---is partially based on the \emph{Simultaneous Resolution Theorem}~\cite[Proposition~5.2.2]{Cartan-Eilenberg}, also known as the \emph{Horseshoe Lemma}~\cite[Lemma 2.2.8]{Weibel}.

\begin{figure}[h]
	{\[
			\xymatrix@C=3.5em{
			\cdots \ar[r] & C_2 \ar@{.>}[ld] \ar@<-.5ex>[d]_-{f_2} \ar@<.5ex>[d]^-{g_2} \ar[r]^-{d^C_2} & C_1 \ar@{.>}[ld]|-{h_1} \ar@<-.5ex>[d]_-{f_1} \ar@<.5ex>[d]^-{g_1} \ar[r]^-{d^C_1} & C_0 \ar@{.>}[ld]|-{h_0} \ar@<-.5ex>[d]_-{f_0} \ar@<.5ex>[d]^-{g_0} \ar[r] & 0 \ar@{.>}[dl]^-(.4){h_{-1}=0} \\
			\cdots \ar[r] & E_2 \ar[r]_-{d^E_2} & E_1 \ar[r]_-{d^E_1} & E_0 \ar[r] & 0
			}
		\]}
	\caption{The lowest degrees of a chain homotopy}
	\label{Fig Chain Intro}
\end{figure}
The abelianness of both categories $\C$ and $\E$, along with the additivity of the functor $F$, are crucial in proving that the derived functors of $F$ are well defined. This proof relies on the concept of a \emph{chain homotopy}. Recall that for a pair of parallel morphisms \(f\), \(g\colon (C,d^C) \to (E,d^E)\) of positively graded chain complexes in an abelian category, a \defn{chain homotopy} from $f$ to $g$ (illustrated in Figure~\ref{Fig Chain Intro}) is a morphism of graded objects \(h\colon (C,d^C) \to (E,d^E)\) of degree $1$ such that
\begin{equation}
	f - g = d^{E} \comp h + h \comp d^{C}
	\label{Abelian-Homotopy}
\end{equation}
holds. Let us outline some key properties of this concept~\cite{Cartan-Eilenberg, MacLane:Homology}.

\begin{theorem}\label{Facts}
	In the abelian setting:
	\begin{enumerate}\label{enum - why derived functors well defined}
		\item Homotopic chain morphisms have the same homology, so chain homotopically equivalent complexes have isomorphic homologies.
		\item Any two projective resolutions of an object are homotopically equivalent.
		\item Additive functors preserve chain homotopies.\noproof
	\end{enumerate}
\end{theorem}

When $\C$ has enough projectives, these properties ensure that the left derived functors of an additive functor $F\colon \C \to \E$ are well defined.

\subsection{Towards a non-additive setting}\label{sec - explanation non additive context}
The context of \emph{semi-abelian categories}, as defined by Janelidze--Márki--Tholen~\cite{Janelidze-Marki-Tholen}, and the more general \emph{homological categories}\footnote{These are not to be confused with the \emph{homological categories} of Grandis~\cite{Grandis-HA2}, also called \emph{ex3-categories}; for instance, the semi-abelian category $\Gp$ of all groups is not homological in Grandis's sense.}, as defined by Borceux--Bourn~\cite{Borceux-Bourn}, provide a highly suitable non-additive framework for homological algebra. Besides the validity of classical homological diagram lemmas~\cite{Borceux-Bourn}, these categories support a well-established homology theory based on simplicially derived functors à la Tierney--Vogel~\cite{Tierney-Vogel2} and Barr--Beck~\cite{Barr-Beck}. This has led to extensive literature on general versions of the Hopf formulae~\cite{EGVdL,Janelidze:Hopf} and interpretations of cohomology groups~\cite{RVdL2,PVdL1}, linked with commutator theory~\cite{RVdL,RVdL3}, which becomes non-trivial in this non-abelian setting.

Chain complexes and their homology behave as expected~\cite{EverVdL2}, but chain resolutions have been largely avoided in the literature. This is despite the fact that they are indeed available in this context, provided the category has enough (regular) projectives. The primary reason for using simplicial techniques in this non-additive context is that the traditional procedure for deriving functors does not work. Our goal here is to explore the conditions under which chain resolutions can be effectively used.

The key problem is that projective resolutions for a given object are not unique, and the definition of a chain homotopy involves the additive structure on hom-sets, making it unsuitable for proving uniqueness up to homotopy in a non-additive context. The primary challenge we face in extending the above definition to a non-additive setting is how to express Equation~\eqref{Abelian-Homotopy}. For instance, the semi-abelian category $\Gp$ of (non-abelian) groups is non-additive precisely because the pointwise sum (or product, depending on the notation used) of two group homomorphisms, although defined as a function, is no longer a homomorphism. Consequently, Equation~\eqref{Abelian-Homotopy} does not make sense as such in the category $\Gp$.

We may, however, view this condition differently, in a way that eventually allows us to use the so-called \emph{subtractive} structure with which each homological category with binary coproducts is naturally equipped. Still in the abelian setting, we can rewrite Equation~\eqref{Abelian-Homotopy} using only a single operation, namely the difference\footnote{Rewriting the equation in terms of just sums seems to lead nowhere---unless, following Grandis~\cite{Grandis-HA2}, we assume that hom-sets carry a commutative monoid structure, which is in our semi-abelian context would amount to abelianness of the category under consideration.}:
\begin{equation}\label{homotopy via differences}
	d^{E}\comp h = (f-g) - h\comp d^{C}.
\end{equation}
This still doesn't make sense as such in a homological category with binary coproducts, since, like the sum of group homomorphisms, the difference of group homomorphisms need not be a homomorphism. However, this viewpoint opens the way to an interpretation inspired by the theory of so-called \emph{approximate subtractions}~\cite{DB-ZJ-2009b, DB-ZJ-2009}. Originally developed in the context of \emph{subtractive categories}~\cite{ZJanelidze-Subtractive, ZJanelidze-Snake}, this theory allows the use of (in some suitable sense) \emph{imaginary} or \emph{approximate} differences even when the hom-sets lack an abelian group structure.

Observe that the difference of two morphisms \( f \), \( g \colon X \to Y \) in an abelian category can be rewritten as
\begin{equation*}
	\vcenter{\xymatrix@!0@R=4em@C=6em{0 \ar[r] & X \ar@{{ |>}->}[r]^-{\univ{1_X}{-1_X}} \ar[rd]_-{f-g} & X \oplus X \ar@{-{ >>}}@<.5ex>[r]^-{\nabla_X} \ar[d]^-{\couniv{f}{g}} & X \ar@{{ >}->}@<.5ex>[l]^-{\iota_1} \ar[r] & 0\\
	&& Y}}
\end{equation*}
where the horizontal sequence is exact and where \( f - g = \couniv{f}{g} \univ{1_X}{-1_X} \) (see for instance \cite{Freyd} for further explanation). We can mimic this procedure in the context of a pointed category: given any object \( X \) for which the needed (co)limits exist, we let \( (D(X), \delta_X) \) be the kernel of \( \nabla_X = \couniv{1_X}{1_X} \colon X + X \to X \) and call it the \defn{difference object} of \( X \).
\begin{equation*}
	\vcenter{\xymatrix@!0@R=4em@C=6em{D(X) \ar@{{ |>}->}[r]^-{\delta_X} \ar[rd]_-{f-g} & X + X \ar@<.5ex>[r]^-{\nabla_X} \ar[d]^-{\couniv{f}{g}} & X \ar@{{ >}->}@<.5ex>[l]^-{\iota_1} \\
	& Y}}
\end{equation*}
If now \( f \), \( g \colon X \to Y \), we view the composite \( \couniv{f}{g} \circ \delta_X \) as a kind of \emph{approximate difference} of \( f \) and \( g \), and write it as \( f - g \colon D(X) \to Y \). Note that even though this definition requires very little from the category in which it is made, its usefulness depends strongly on the category's properties. For our purposes, it will be convenient to work in a homological category with coproducts, which encompasses the subtractive setting of~\cite{DB-ZJ-2009}--- see Section~\ref{sec - approximate difference}.

In Section~\ref{sec - approx homotopy}, we will explain how Equation~\eqref{homotopy via differences}, interpreted in terms of approximate differences, naturally leads to the concept of \emph{approximate chain homotopy}. For these chain homotopies, a general version of item (1) in Theorem~\ref{Facts} holds. Specifically, in a homological category $\C$ with binary coproducts, two chain morphisms that are homotopic will always have the same homology. Furthermore, a simplicial homotopy between two simplicial morphisms will be shown to induce a chain homotopy, in our sense, between the corresponding chain morphisms of Moore chain complexes (Proposition~\ref{thm - simplicial homotopy implies chain homotopy - nonadd case}).

Unfortunately, this does not suffice for item (2) of Theorem~\ref{Facts} to be valid in general. For this, we need projective objects to interact well with the split short exact sequences in the category $\C$. Note that in an additive category with finite limits, whenever we have a split short exact sequence
\begin{equation*}\label{Split Extension}
	\xymatrix{0 \ar[r] & K \ar@{{ |>}->}[r]^-k & X \ar@{-{ >>}}@<.5ex>[r]^-{f} & Y \ar[r] \ar@{{ >}->}@<.5ex>[l]^-{s} & 0,}
\end{equation*}
if the middle object \(X\) is projective, then so are \(K\) and \(Y\). Indeed, both are retracts of \(X\), because \(X\) is the biproduct \(K \oplus Y\). In a homological category, \(Y\) will remain a retract of \(X\), so it is still projective if \(X\) is projective. However, for \(K\) to be projective is a non-trivial condition, which may or may not hold depending on the situation. We say that the class of projectives in \(\C\) is \defn{closed under protosplit subobjects} or that \(\C\) \defn{satisfies Condition (P)}. As it turns out, a homological category with enough regular projectives satisfies Condition \P\ if and only if a variation on the Horseshoe Lemma holds (see Section~\ref{sec - non add Horseshoe lemma}). Note that a version of the Horseshoe Lemma is needed anyway if we wish to obtain a long exact homology sequence relating the derived functors of the objects in a given short exact sequence. Such a long exact sequence is indeed available—see Section~\ref{sec - syzygy}.

Item (3) in Theorem~\ref{Facts}---preservation of chain homotopies---requires \(F\) to have certain properties that, in the abelian case, correspond to \(F\) being an additive functor. Specifically, \(F\) must preserve binary coproducts, \defn{proper morphisms} (those morphisms whose image is a \defn{normal subobject}, i.e., a kernel of some morphism), and split short exact sequences. This is sufficient for compatibility with the simplicial approach (Theorem~\ref{Thm Compa}). When \(\E\) is an abelian category, preservation of proper morphisms is automatic; and when \(\C\) is also abelian, these conditions are equivalent to \(F\) being additive.

The condition that \(F\) is a right exact functor, common in the abelian setting, can be captured in our more general framework by requiring it to be \defn{sequentially right exact}---sending short exact sequences to right exact sequences---while also preserving binary coproducts and protosplit monomorphisms (i.e., the kernel of a retraction is sent to a monomorphism). Under these conditions, we obtain results such as the long exact homology sequence (Theorem~\ref{them sec - long exact sequence of derived functors intro}), the Syzygy Lemma (Theorem~\ref{Syzygy Lemma}).

\section{Self-contained statement of the main results}\label{sec - SA vs AB}
This article extends the notion of a derived functor, typically defined for an additive functor between abelian categories, to a non-additive context. Our objective is to derive functors between (Janelidze--Márki--Tholen~\cite{Janelidze-Marki-Tholen}) semi-abelian categories in the same manner as in the abelian setting: by applying the functor to a projective resolution and then taking homology.

However, the absence of an additive structure on the categories complicates this process, since the conventional arguments ensuring that these functors are well defined do not hold in this setting. Consequently, simplicial methods are traditionally employed instead of chain resolutions~\cite{EverVdL2, EGVdL, RVdL2}. This article shows that, under certain conditions, chain resolutions do work after all.

Our aim in the current section is to present the main theorems of the article in a clear and essentially self-contained manner. These theorems can be understood without prior knowledge of approximate homotopies or expertise in the semi-abelian context. Any less familiar terminology used here is explained in Subsection~\ref{Subsec Context} below.

\subsection{Three main results}
The first result extends the definition of the left derived functors of an additive functor between abelian categories to a non-additive setting:

\begin{thm}
	Let $\C$ be a pointed regular category with binary coproducts and with enough projectives which satisfies Condition \P. Let $\E$ be a homological category with binary coproducts. Let $F\colon \C\to \E$ be a protoadditive functor that preserves proper morphisms and binary coproducts. Then the formula
	\[
		\Left_n(F)(X)\coloneq \H_n(F(C(X)))
	\]
	where $X$ is an object of $\C$ and $C(X)$ is a chosen projective chain resolution of $X$, defines a functor $\Left_n(F)\colon\C\to \E$ for each $n\in \ZZ$.
\end{thm}

In other words, the left derived functors of $F$ are defined as in the abelian context. The proof follows steps (1)--(3) of Theorem~\ref{Facts}, adapted for the notion of chain homotopy developed in this article. This is an application of Theorem~\ref{Thm Def L}.

In fact, when moreover $\C$ is protomodular (so that now both $\C$ and $\E$ are homological categories), the condition that $F$ is protoadditive is strong enough for the $\Left_n(F)$ to coincide with the derived functors obtained via projective simplicial resolutions, once those exist. In particular, we find the following result, which in the text occurs as Theorem~\ref{Thm Compa}:

\begin{thm}
	Let $\C$ be a semi-abelian variety (in the sense of universal algebra, as in~\cite{Bourn-Janelidze}) for which Condition \P\ holds, with comonad $\G$ induced by the forgetful adjunction to $\Set$. Let $\E$ be a semi-abelian category and $F\colon \C\to \E$ a protoadditive functor which preserves proper morphisms and binary coproducts. Then $\Left_n(F)$ is naturally isomorphic to $\H_{n+1}(-,F)_{\G}$, the $(n+1)$th simplicially derived functor of~$F$ in the sense of Barr--Beck~\cite{Barr-Beck,EverVdL2}.
\end{thm}

The reason for this is two-fold: first of all, any two simplicially homotopic morphisms of simplicial objects induce homotopic chain morphisms between the Moore normalizations---this is Theorem~\ref{thm - simplicial homotopy implies chain homotopy - nonadd case}. Note that this works despite the fact that the Dold-Kan equivalence (between the categories of simplicial objects and positively graded chain complexes) is not available here. Secondly, the normalization of a simplicial resolution is always a projective chain resolution, by closedness of the class of projectives under protosplit subobjects (Condition \P) and the fact that in a homological category, exactness of a simplicial object agrees with exactness of its normalization~\cite{Tierney-Vogel2,GVdL,Goedecke}. We require Barr exactness of the category $\E$ to guarantee that the normalization of any simplicial object is a proper chain complex, which is essential for the homology objects to detect exactness~\cite{VdLinden:Doc}.

A third key result mimics what happens for a right exact functor between abelian categories, Theorem~\ref{them sec - long exact sequence of derived functors intro} in the text:

\begin{thm}
	Let $\C$ be a homological category with binary coproducts and with enough projectives which satisfies Condition \P. Let $\E$ be a homological category with binary coproducts. Let $F\colon \C\to \E$ be a sequentially right-exact functor which preserves protosplit monomorphisms and binary coproducts. Any short exact sequence
	\[
		\xymatrix{0\ar[r] & K\ar@{{ |>}->}[r]^-k & X \ar@{-{ >>}}[r]^-{f} & Y\ar[r] & 0}
	\]
	in $\C$ gives rise to a long exact sequence in $\E$ as in Figure~\ref{fig long exact sequence with F} which depends naturally on the given short exact sequence.
\end{thm}
\begin{figure}
	$\xymatrix@C=4em{
		\\
		& & \vdots \ar`d[l]`[dll]_-{\delta_{n+1}}[dll]  & \\
		\Left_n(F)(K)\ar[r]^-{\Left_n(F)(k)}& \Left_n(F)(X) \ar[r]^-{\Left_n(F)(f)} & \Left_n(F)(Y)\ar`d[l]`[dll]|-{\delta_n}[dll] &\\
		\Left_{n-1}(F)(K)\ar[r]& \cdots\ar[r] & \Left_{1}(F)(Y)\ar`d[l]`[dll]|-{\delta_1}[dll] & \\
		F(K)\ar[r]^-{F(k)}& F(X)\ar@{-{ >>}}[r]^-{F(f)} & F(Y)\ar[r] & 0\\
		}$
	\caption{Long exact sequence involving the derived functors of $F$} \label{fig long exact sequence with F}
\end{figure}
The conditions on $\C$ are satisfied, for instance, by the variety of Lie algebras over any commutative unitary ring (\ref{Thm Lie Satisfies (P)}), by crossed modules (\ref{Thm Crossed Modules}) and more generally by \ncat-groups.

If the derived functors $\Left_n(F)$ of a functor $F$ are not additive, like the title of this article indicates, then what properties do they have? We prove in Section~\ref{Section Derived Functors} that they are protoadditive functors, which means that they preserve semi-direct products, provided $F$ is a functor between semi-abelian categories which is itself protoadditive and preserves binary coproducts and proper morphisms (Corollary~\ref{Protoadditive between semi-abelian cats}).

\subsection{The context}\label{Subsec Context}
We now explain the terminology in the statements of those results.

A category is \defn{pointed} when it has an object that is initial and terminal at the same time. In such categories, we can define \defn{protosplit subobjects}: such is a subobject $K\leq X$ which is represented by a kernel of a split epimorphism with domain $X$.

A \defn{normal} category in the sense of~\cite{ZJanelidze-Subtractive} is a pointed regular~\cite{Barr-Grillet-vanOsdol} category where every regular epimorphism is a normal epimorphism, i.e., a cokernel of some morphism.

We mainly work in the context of \emph{homological categories}. A category $\C$ is called \defn{(Borceux--Bourn) homological}~\cite{Borceux-Bourn} if it is pointed, regular and (Bourn) protomodular~\cite{Bourn1991}. The protomodularity condition says that the \emph{Split Short Five Lemma} holds in $\C$, or equivalently, that whenever we have a split epimorphism with its kernel as in
\[
	\xymatrix{K \ar@{{ |>}->}[r]^-k & X \ar@<.5ex>[r]^-{f} & Y\text{,} \ar@{{ >}->}@<.5ex>[l]^-{s}}
\]
the object $X$ is ``covered'' or ``generated'' by the outer objects $K$ and $Y$ in the sense that the inclusions $k$ and $s$ (where $k=\ker(f)$ and $s$ splits $f$) are \emph{jointly extremally epimorphic}, which means that they do not both factor through a proper subobject of~$X$. It implies, for instance, that each regular epimorphism is a normal epimorphism. In other words, the category in question is normal.

A \defn{(Janelidze--Márki--Tholen) semi-abelian category}~\cite{Janelidze-Marki-Tholen} is a category which is pointed, (Barr) exact and (Bourn) protomodular, and which has binary coproducts. This stricter definition ensures, for instance, good behaviour of simplicial objects. (In particular, the normalization of any simplicial object is a \emph{proper} chain complex, which means that the images of the differentials are normal subobjects.)

Examples include any abelian category, the category of groups $\Gp$, and more generally any variety of $\Omega$-groups~\cite{Higgins}, of which by definition the signature admits a group operation and a unique constant---such as any variety $\LieK$ of Lie algebras over a ring $\K$ or any variety of non-associative algebras over $\K$~\cite{VdL:nonassociativealgebra}. An important example of a homological category which is not semi-abelian is the category of topological groups. The category $\OrdGrp$ of preordered groups is known to be normal but not homological~\cite{MMClementinoNMartinsFerreiraAMontoli2019-Preordered}: a preordered group being a group equipped with a preorder (a reflexive and transitive relation) for which the group operation is monotone; arrows are monotone group homomorphisms.

Some of the above examples satisfy Condition \P. Indeed, it will be automatic for abelian categories; it will hold for $\Gp$ by the Nielsen--Schreier theorem; for $\LieK$ (where $\K$ is a commutative ring) it is a non-trivial result (see Section~\ref{sec - Lie algebra P condition}); while in the category of associative algebras over a field, for instance, it is still an open question.

A key concept when we are dealing with derived functors is the notion of a \emph{projective chain resolution}. Let $\C$ be a pointed category with kernels and let~$X$ be an object of~$\C$. A~\defn{projective chain resolution of $X$} consists of a positively graded chain complex $(C,d)$ as in Figure~\ref{Figure Projective Resolution}, all of whose objects $C_i$ are projective (with respect to the class of regular epimorphisms) together with a regular epimorphism $d_0\colon C_0\to X$ such that for each $n\geq 0$, the lifting $\bar{d}_{n+1}\colon {C_{n+1}\to \Ker(d_n)}$ of $d_{n+1}$ over the kernel $\Ker(d_n)$ of $d_n$ is a regular epimorphism. Once we assume that the category $\C$ has enough projectives, recursively constructed projective chain resolutions exist for any object of the category.

In the context of a normal category, a resolution of $X$ is an exact sequence as in Figure~\ref{Figure Projective Resolution Intro}. Exactness in $C_n$, for instance, means that the morphism $\bar{d}_{n+1}$ induced by the kernel of $d_n$ is a normal epimorphism, which is equivalent to saying that the image of $d_{n+1}$ is the kernel of $d_n$. Here, the interpretation in terms of homology as in the figure's caption makes sense.

In a pointed category, a morphism is \defn{proper} if it can be factorized as a regular epimorphism followed by a normal monomorphism. In a normal category, this means that the monic part of the image factorization is a kernel, while the epic part is a cokernel. A functor $F\colon {\C \to \E}$ between between pointed categories with kernels and cokernels $\C$ and~$\E$ will be called a \defn{sequentially right-exact functor}, as in ~\cite{PVdL1}, if it sends an exact sequence
\[
	\xymatrix{K \ar[r]^-{k} & A \ar@{-{ >>}}[r]^-{f} & B \ar[r] & 0,}
\]
so a proper morphism $k$ with its cokernel $f$, to the exact sequence
\[
	\xymatrix{F(K) \ar[r]^-{F(k)} & F(A) \ar@{-{ >>}}[r]^-{F(f)} & F(B) \ar[r] & 0.}
\]
These functors are meant to replace the right exact functors from the abelian context.

If we want to mimic additive functors outside the context of additive categories, we can use \emph{protoadditive functors}~\cite{Everaert-Gran-TT,EG-honfg}. A functor $F\colon \C \to \E$ between pointed categories $\C$ and $\E$ is \defn{protoadditive} if it preserves kernels of split epimorphisms.

\subsection{The abelian setting}\label{Subsec Abelian Setting}
Coming back to the above theorems, we see that the categories $\C$ and $\E$ satisfy conditions which are weaker than abelianness. On the other hand, we impose seemingly stronger conditions on the functors to compensate for the weaker context. In the first theorem, any functor between abelian categories which preserves binary coproducts is additive. Therefore, the assumption of being protoadditive is redundant in the abelian setting. And in the third one, in the abelian setting, we only require right exactness of the functor. Here we have to impose two non-trivial assumptions---preservation of protosplit monomorphisms and the preservation of binary coproducts---which are redundant in the abelian case, but just like protoadditivity turn out to be relatively strong in the non-additive setting.

\section{Closedness of projectives under protosplit subobjects}\label{sec - condition P}
As recalled in the introduction, the first step in the definition of the derived functors of a functor is the construction of a projective resolution for any given object. We need to prove that any two such resolutions are chain homotopically equivalent. Crucial for the existence of such a chain homotopy equivalence is Condition \P\ that \emph{the class of projective objects is closed under protosplit subobjects}. The aim of the present section is to explain the details, necessary to understand what this condition means. The actual proof that under Condition \P, a chain homotopy equivalence can indeed always be constructed is the subject of Section~\ref*{sec - approx homotopy}.

\subsection{Projective and free objects}\label{sec - def projective and free objects}
We start by recalling what are projective objects (relative to the regular epimorphisms) in a category.

\begin{definition}\label{def_projective}
	An object $P$ of a category $\C$ is said to be \defn{(regular) projective} when, given a regular epimorphism $p\colon X\twoheadrightarrow Y$ and a morphism $f\colon P\to Y$ there exists a lifting $g\colon P\to X$ of $f$ over $p$:
	\[
		\xymatrix{
		& X\ar@{>>}[d]^-{p} \\
		P \ar@{-->}[ur]^-{g} \ar[r]_-{f} & Y
		}
	\]
	a morphism $g$ such that $pg=f$.

	A category has \defn{enough (regular) projectives} when each object $A$ is a regular quotient of a projective object: there exists a projective object $P$ together with a regular epimorphism $p\colon P \twoheadrightarrow A$.
\end{definition}

\begin{remarks}
	\begin{enumerate}
		\item Outside the abelian setting, the classes of \emph{epimorphisms} and \emph{regular epimorphisms} often do not coincide. For instance, in the variety of unitary rings, the inclusion $i\colon\mathbb{Z}\to \mathbb{Q}$ is an epimorphism and not a surjection. However, in any variety of algebras, the regular epimorphisms are precisely the surjective morphisms of algebras.
		\item It is easy to see that in any category, the class of projectives is closed under coproducts and under retracts. \label{rectract_projective and coproduct}
		\item Recall that a category with finite limits and coequalizers of kernel pairs is \defn{regular} when its regular epimorphisms are pullback-stable. It is known that existence of enough projectives in $\C$ implies regularity of~$\C$, as soon as the needed (co)limits exist~\cite{Quillen,VdLinden:Simp}.
		\item It is easy to see that an object $P$ in a regular category is projective if and only if each regular epimorphism $X\twoheadrightarrow P$ is a split epimorphism.
	\end{enumerate}
\end{remarks}

In a variety of algebras, a convenient class of projective objects is the class of free objects\label{free implies projective}: objects which belong to the image of the left adjoint $T$ to the forgetful functor $U$ to $\Set$. This follows easily from the definition. In certain varieties---see Example~\ref{ex_schreier variety} below---the two classes coincide, but in general, the converse does not hold. However, in a variety of algebras, an object is a projective object if and only it is a retract of a free object.

\subsection{On Condition \P}

The proof of the ``classical'' Horseshoe Lemma in abelian categories (\cite[Proposition~V.2.2]{Cartan-Eilenberg}, \cite[Lemma 2.2.8]{Weibel}) depends on the class of projective objects being closed under biproducts (see for instance~\cite[Proposition 4.6.3]{Borceux:Cats1}). In our present context (where biproducts are not available) this is replaced by closedness under coproducts combined with the following condition, which in the context of a homological category we will show to be equivalent to a non-additive version of the Horseshoe Lemma---see Section~\ref{sec - non add Horseshoe lemma}.

\begin{definition}\label{Def Protosplit Subobject}
	In a pointed category, a subobject $K\leq X$ is called a \defn{protosplit subobject} when it is represented by a kernel of a split epimorphism: there exist $f\colon X\to Y$, $s\colon Y\to X$ and a representing monomorphism $k\colon K\to X$ such that $f\comp s=1_Y$ and $k=\ker(f)$.
\end{definition}

\begin{remark}
	In the setting of a normal category (where every regular epimorphism is normal), this means there exists a split short exact sequence
	\begin{equation}\label{SSES}
		\xymatrix{0 \ar[r] & K \ar@{{ |>}->}[r]^-k & X \ar@{-{ >>}}@<.5ex>[r]^-{f} & Y \ar[r] \ar@{{ >}->}@<.5ex>[l]^-{s} & 0}
	\end{equation}
	whose kernel part is the representing monomorphism $k\colon K\to X$.
\end{remark}

\begin{definition}\label{Def Condition (P)}
	We say that a pointed category satisfies \defn{Condition (P)} when \emph{the class of projective objects is closed under protosplit subobjects}: given a protosplit subobject $K\leq X$, if $X$ is projective, then $K$ is projective.
\end{definition}

\begin{remark}
	By the closedness under retracts of the class of projectives, we know that the object $Y$ that occurs in Definition~\ref{Def Protosplit Subobject} is projective as well, as a retract of the projective object $X$.
\end{remark}

\begin{example}\label{ex_ab_cat_P_condition}
	Each additive category with finite limits, and thus also each abelian category, satisfies Condition \P. It is indeed well known that here, the object $K$ in~\eqref{SSES} is a retract of $X$, via the isomorphism of split short exact sequences involving $X\cong K\oplus Y$. So the claim follows from the closedness of the class of projective objects under retracts.
\end{example}

\begin{example}\label{ex_schreier variety}
	A variety of algebras is called a \defn{Schreier variety} when a subobject of a free object is again free. Clearly, all Schreier varieties satisfy Condition \P. Indeed, any projective object is a subobject of a free object and so ``projective = free'' in this context.

	Many examples are known; here are some selected references: groups~\cite{Johnson_DL,Schreier,Nielsen}, Lie algebras over a field~\cite{Shirshov,Witt1956,Umirbaev}, non-associative algebras over a field~\cite{Kurosh:art,Umirbaev}, commutative and anti-commutative algebras over a field~\cite{Shirshov1954,Umirbaev}, Lie $p$-algebras over a field~\cite{Umirbaev,Witt1956}, Lie superalgebras over a field~\cite{Umirbaev,Mikhalev,Shtern}, pre-Lie algebras over a field of characteristic zero~\cite{DotsenkoUmirbaev}, Akivis algebras over a field~\cite{ShestakovUmirbaev}, modules over a principal ideal domain~\cite[Chapter~1, Theorem~5.1]{Hilton-Stammbach}, so in particular abelian groups~\cite[Section~1.2]{Johnson_DL}, which are modules over the ring of integers.

	Note that a variety of modules need not be Schreier in general, even through by Example~\ref{ex_ab_cat_P_condition}, it always satisfies Condition \P.
\end{example}

We shall later see that, just as in the abelian case, if the domain of a functor $F$ is a Schreier variety, then each of its derived functors $\Left_n(F)$ for $n\geq 2$ is trivial (Corollary~\ref{Cor Syzygy Schreier}). So while Schreier varieties do satisfy Condition \P, they are not very interesting from a homological-algebraic point of view. This means we have to focus on non-Schreier varieties satisfying \P. It seems, however, that Condition \P\ has not yet been studied at all in the literature, so currently examples are scarce. We present a non-trivial one in Section~\ref{sec - Lie algebra P condition}, and another one in Section~\ref{Sec:CrossedModules}.

Varieties of non-associative algebras over a field need not be Schreier:

\begin{example}
	Kurosh proved in~\cite{Kurosh:art} that the free associative algebra with one generator already contains subalgebras which are not free. His construction also shows that other types of algebras over a field are not Schreier: alternative, right- or left-alternative, or Jordan algebras over a field.
\end{example}

\begin{remark}
	In~\ref{ex_schreier variety}, the last two examples are abelian categories---where \P\ is automatically satisfied by Example~\ref{ex_ab_cat_P_condition}. Note, however, that there do of course exist categories of modules over a ring which are not Schreier.
\end{remark}

\subsection{Characterization in terms of free objects}
We now investigate what happens with Condition \P, if we assume that the middle object $X$ in \eqref{SSES-Non-Normal} is free rather than just projective. \emph{A priori}, what we find is a weaker condition. However, it turns out that in any pointed Mal'tsev variety---which~\cite{Maltsev} admits a ternary operation $p$ such that $p(x,x,z)=z$ and $p(x,z,z)=x$---hence \emph{a fortiori} in any homological variety, the two are equivalent.

\begin{proposition}\label{propo_equivalent_conditions}
	Let $\V$ be a pointed Mal'tsev variety of algebras, with forgetful functor $U\colon \V\to \Ens$ and its left adjoint $T\colon \Ens\to \V$. The following conditions are equivalent:
	\begin{tfae}
		\item for any protosplit subobject $K\leq X$ as in Definition~\ref{Def Protosplit Subobject}, if $X$ is projective then $K$ is projective;
		\item for any protosplit subobject $K\leq X$, if $X$ is free then $K$ is projective;
		\item for any split epimorphism $g\colon A \to B$ in $\Ens$, the kernel of $T (g)$ is projective.
	\end{tfae}
\end{proposition}

\begin{proof}
	Since, in a variety of algebras, any free object is projective, (i) implies (ii). The split epimorphism of sets $g$ in the statement of (ii) induces a diagram
	\begin{equation}\label{SSES-Non-Normal}
		\xymatrix{0 \ar[r] & K \ar@{{ |>}->}[r]^-k & X \ar@{-{>>}}@<.5ex>[r]^-{f} & Y \ar@{{ >}->}@<.5ex>[l]^-{s}}
	\end{equation}
	as in Definition~\ref{Def Protosplit Subobject}, where $X=T (A)$ and $Y=T (B)$. Since $T (A)$ is a free object on $A$, we have that (ii) implies (iii).

	We now focus on the proof that (iii) implies (i). So consider a split epimorphism $f$ with section $s$ and kernel $k$ as in~\eqref{SSES-Non-Normal}. Covering each object with a free object, we obtain the diagram in Figure~\ref{diagm: (ii) implies (i)}, where the horizontal sequences are split epimorphisms with their kernel, $T U (X)$ and $T U (Y)$ are free objects and $\varepsilon_X$ and $\varepsilon_Y$ are regular epimorphisms.
	\begin{figure}
		%\resizebox{\textwidth}{!}
		{
		\(
		\xymatrix@=3.5em{
		0\ar[r]  & K' \ar@{{ |>}->}[r]^-{k'} \ar@{-->}[d]_-{\beta}  & T U (X) \ar@{>>}@<.5ex>[r]^-{T U (f)} \ar@{>>}[d]_-{\varepsilon_X} & T U (Y) \ar@{{ >}->}@<.5ex>[l]^-{T U (s)} \ar@{>>}[d]^ -{\varepsilon_Y}  \\
		0 \ar[r] & K \ar@{{ |>}->}[r]_-{k} & X \ar@{>>}@<.5ex>[r]^-{f} & Y \ar@{{ >}->}@<.5ex>[l]^-{s}
		}
		\)
		}
		\caption{The implication (iii) $\To$ (i)} \label{diagm: (ii) implies (i)}
	\end{figure}
	The naturality of $\varepsilon$ implies that $f\comp \varepsilon_X=\varepsilon_Y \comp T U (f)$ and hence we have the morphism $\beta$ determined by the universal property of the kernel of~$f$.  We claim that $\beta$ is a split epimorphism. Then $K$ will be a retract of $K'$, which is a projective object by (iii) applied to the split epimorphism of sets $U(f)\colon U(X)\to U(Y)$.

	Let us prove our claim. Since the codomains of $\varepsilon_X$ and $\varepsilon_Y$ are projective objects, they are split epimorphisms. Since $\V$ is a regular Mal'tsev category, the right-hand square is a regular pushout~\cite[Proposition 3.2]{Bourn2003}, which means that the comparison $\univ{\varepsilon_X}{TU(f)}$ from $TU(X)$ to the pullback $X\times_UTU(Y)$ in Figure~\ref{fig - regular pushout (ii) implies (i)} is a regular epimorphism.
	\begin{figure}
		%\resizebox{\textwidth}{!}
		{
		\(		\xymatrix{
		T U (X) \ar@{-->>}[dr]|-{\univ{\varepsilon_X}{TU(f)}} \ar@{-{>>}}@<0.5 ex>[rr]^-{T U (f)} \ar@{>>}[dd]^-(.3){\varepsilon_X} && T U (Y) \ar@{{ >}->}@<0.5 ex>[ll]^-{T  U (s)} \ar@{>>}[dd]_ -{\varepsilon_Y}   \\
		& X\times_{Y}T U (Y) \ar@{-{>>}}[ur]_ -{f'}  \ar@{>>}[dl]^-(.3){\varepsilon_Y'} &\\
		X \ar@/^ 1.5pc/@{-->}[uu]^ -{j_X}  \ar@/^1pc/@{-->}[ur]^-{j_Y'} \ar@{-{>>}}@<0.5 ex>[rr]^-{f}  && Y  \ar@{{ >}->}@<0.5 ex>[ll]^-{s} \ar@/_1.5pc/@{-->}[uu]_-{j_Y}
		}
		\)}
		\caption{Constructing a splitting for $\beta$} \label{fig - regular pushout (ii) implies (i)}
	\end{figure}
	Let $j_Y\colon Y\to T U(Y)$ be a section of $\varepsilon_Y$. By the pullback property we obtain a section $j'_Y$ of $\varepsilon_Y'$ such that $j_Y\comp f = f' \comp j_Y'$.

	Since $X$ is projective, we can define a section of $\varepsilon_X$ as a lifting of $j_Y'$ over the regular epimorphism $\univ{\varepsilon_X}{TU(f)}$, so that $\univ{\varepsilon_X}{TU(f)}\comp j_X= j_Y '$. This morphism $j_ X\colon X\to T U(X)$ is indeed a section of $\varepsilon_X$, since $\varepsilon_X\comp j_X=\varepsilon_Y' \comp \univ{\varepsilon_X}{TU(f)} \comp j_X=\varepsilon_Y' \comp j_Y ' = 1_X$. We see that
	$$T U(f)\comp j_X=f'\comp \univ{\varepsilon_X}{TU(f)} \comp j_X=f' \comp j_Y' = j_Y\comp f.$$ By the universal property of $\kker (T  U (f))$, there exists a morphism $j_K\colon {K\to K'}$ such that $\kker (T U(f)) \comp j_K = j_X \comp k $. Now by definition of $\beta$ and $j_K$ we have
	$$k \comp 1_K = 1_X \comp k = \varepsilon_X \comp j_X \comp k = \varepsilon_X \comp k' \comp j_K=k \comp \beta \comp j_K $$
	and since $k$ is monic this implies that $\beta \comp j_K = 1_K$, proving $\beta$ to be a split epimorphism.
\end{proof}

Proposition~\ref{propo_equivalent_conditions} will help us in the next section, where we give a class of examples of semi-abelian varieties which need not be Schreier but do satisfy~\P.

\section{Example: Lie algebras}\label{sec - Lie algebra P condition}
Our aim here is to prove that the semi-abelian variety $\LieK$ of Lie algebras over any commutative unitary ring $\K$ satisfies Condition \P.

\subsection{Free Lie algebras}
It is well known---see for instance~\cite[Section 0.2]{Reutenauer}---how to construct the free Lie algebra $T(X)$ for a given set $X$: the elements of $T(X)$ are the (non-commutative, non-associative) polynomials with no constant term and variables in $X$ such that the binary product, which is simply a linear extension of the concatenation of monomials with variables in $X$, satisfies the axioms of a Lie bracket.

Lie algebras over a field form a Schreier variety. The original proof due to Shirshov~\cite{Shirshov} uses the following strategy. First find a set which generates the given subalgebra of a given free Lie algebra; then prove that this subalgebra is \emph{freely} generated---see for instance~\cite[Theorem 2.5]{Reutenauer} for a proof written in English. The assumption that the coefficients ring is a field is a key assumption in the first step of the proof. Indeed, we need the modules involved in this construction to always admit a base---which is false when the ring is not a field. For this reason, the proof does not work for unital commutative rings, as was already noticed by Shirshov in~\cite{Shirshov}. Using an idea of Bahturin in~\cite{Bahturin}, whenever $\K$ is not a field we can build an example of a subalgebra of a free $\K$-Lie algebra which is no longer free.

\begin{example}[An explicit construction of a non-free subalgebra]\label{ex: Lie_K_not_Sch_var PID}
	Let us consider a unital commutative ring $\K$ and a set $X$ with $m$ elements, where $m>1$. Consider the free Lie algebra $T(X)$ and its subalgebra $H\coloneq \alpha T(X)$, where $\alpha\in \K$ is a non-invertible element of $\K$, different from $0$. By definition of $H$, the subalgebra admits a natural grading decomposition $H=\alpha T(X)_1\oplus \alpha T(X)_2\oplus \cdots $ inherited from the grading of $T(X)$.

	Let us assume that $H$ is a free Lie algebra on a set $Y$. Then its abelianization $H^{\mathrm{ab}}\coloneq H/[H,H]$ is isomorphic to the free $\K$-module $\bigoplus_{y\in Y}\K y$, by composition of left adjoints. However, by the grading decomposition of $H$, its abelianization $H^{\mathrm{ab}}$ contains
	\begin{align*}
		\alpha T_2(X)/[\alpha T_1(X),\alpha T_1(X)] \cong \begin{cases}
			                                                  T_2(X)/\alpha T_2(X) & \text{if $\alpha^2\neq 0$,} \\
			                                                  \alpha T_2(X)        & \text{otherwise.}
		                                                  \end{cases}
	\end{align*}
	In both situations, there are non-trivial torsion elements in the module $H^{\mathrm{ab}}$---which contradicts the fact that it is a free module. This explains why $H$ cannot be free.
\end{example}

\subsection{Lazard's elimination theorem}
We fix a commutative unital ring $\K$ and work towards a proof that Condition \P\ holds in the category $\LieK$. We can consider the split short exact sequence
\[
	\xymatrix{0 \ar[r] & K \ar@{{ |>}->}[r]^-k & T(X) \ar@{-{ >>}}@<.5ex>[r]^-{T(f)} & T(Y) \ar[r] \ar@{{ >}->}@<.5ex>[l]^-{T(s)} & 0}
\]
where $(f,s)$ is any split epimorphism $f$ with a section $s$ in $\Set$. We claim that, in this situation, we can prove that $K$ is a retract of a free Lie algebra and so a projective object. Thanks to Proposition~\ref{propo_equivalent_conditions}, we may then conclude that $\LieK$ satisfies Condition \P. The proof of this claim is based on the next result called \defn{Lazard's elimination theorem}.

\begin{theorem}\cite{Lazardinco}\label{thm : Lazard elimination}
	Let $T(X)$ be the free Lie algebra over a commutative unital ring $\K$ and let $A$ and $B$ two disjoint sets such that $X=A\cupdot B$. Then we have an isomorphism
	\[
		T(X)\cong T(M(B) \times A) \rtimes_{\phi} T(B)
	\]
	where $M(B)$ is the free monoid over $B$ and $\phi$ is an appropriate action of $T(B)$ on $T(M(B) \times A)$.
\end{theorem}
\begin{proof}
	See for instance Proposition 10 in Section 2.9 of~\cite{BourbakiLieAlgPart1} for a proof.
\end{proof}

The action $\phi$ may be codified as a Lie algebra homomorphism
\[
	\phi\colon T(B)\to \Der(T(M(B)\times A))
\]
whose codomain is the Lie algebra of all derivations of the free algebra $T(M(B)\times A)$.

The main idea behind the proof of this theorem is to consider the split short exact sequence
\[
	\xymatrix{0 \ar[r] & K \ar@{{ |>}->}[r]^-k & T(A\cupdot B) \ar@{-{ >>}}@<.5ex>[r]^-{g} & T(B) \ar[r] \ar@{{ >}->}@<.5ex>[l]^-{t} & 0}
\]
where $g$ is such that $a\in A\mapsto 0 $ and $b\in B\mapsto b$ and $t$ its canonical splitting ($b\in B\mapsto b$). We then notice that $K$, as a kernel of $g$, not only \emph{contains} the element $a$, as well as its image through (composites of) all inner derivations induced by elements of $B$ (so $\mathrm{ad}_b(a)\coloneq[b,a]$ and $\mathrm{ad}_{b_1}\cdots \mathrm{ad}_{b_n}(a) = [b_1,\dots ,[b_n,a]\cdots]$ where $b\in B$ and $b_i\in B$ for all $1\leq i\leq n$); but in fact, those elements generate $K$. It is, however, not clear at first whether $K$ is \emph{freely} generated by those elements. The theorem tell us that this is indeed the case. The proof depends on finding a suitable action $\phi$ of $T(B)$ on ${\Der(T(M(B) \times A))}$, taking into account that~$K$ is closed under inner derivations induced by the elements of $B$.

We translate the semi-direct product of Theorem~\ref{thm : Lazard elimination} into the split short exact sequence
\[
	\xymatrix{0 \ar[r] & T(M(B)\times A) \ar@{{ |>}->}[r] & T(M(B)\times A) \rtimes_{\phi} T(B) \ar@{-{ >>}}@<.5ex>[r]^-{g} & T(B) \ar[r] \ar@{{ >}->}@<.5ex>[l]^-{t} & 0.}
\]

\begin{theorem}\label{Thm Lie Satisfies (P)}
	For $\K$ any commutative ring with unit, the semi-abelian variety $\LieK$ satisfies Condition \P.
\end{theorem}
\begin{proof}
	Let $f\colon X \to Y$ be a split epimorphism in $\Ens$ and $s$ one of its sections. In order to use Lazard's elimination theorem, we define
	\[
		A\coloneq \{x-sf(x)\mid x\in X\}\subseteq T(X)
		\qquad\text{and}\qquad
		B\coloneq \{s(y)\mid y\in Y\}\subseteq T(X).
	\]
	In the diagram
	\[
		\xymatrix{
		T(X) \ar@<.5ex>[d]^-{\phi} \ar@{-{ >>}}[r]^-{T(f)} & T(Y) \ar@<.5ex>[d]^-{T(s|^B)} \\
		T(A\cupdot B) \ar@<.5ex>[u]^-{\psi} \ar@{-{ >>}}[r]_-{g} & T(B) \ar@<.5ex>[u]^-{T(f|_B)}
		}
	\]
	the morphism $g$ sends generators in $A$ to zero and generators in $B$ to themselves, while $\phi$ and $\psi$ are defined as
	\[
		\phi(x)=(x-sf(x))+sf(x)
	\]
	and
	\[
		\psi(x-sf(x))=x-sf(x),\qquad \psi(s(y))=s(y)
	\]
	for $x\in X$ and $y\in Y$. It represents a vertical, upward-pointing split epimorphism with downward-pointing section. Taking kernels, via Theorem~\ref{thm : Lazard elimination} we find a vertical, upward-pointing split epimorphism of short exact sequences as in
	\[
		\xymatrix{
		0 \ar[r] & \Ker(T(f)) \ar@{{ |>}->}[r] \ar@<.5ex>@{.>}[d] & T(X) \ar@<.5ex>[d]^-{\phi} \ar@{-{ >>}}[r]^-{T(f)} & T(Y) \ar@<.5ex>[d]^-{T(s|^B)} \ar[r] & 0\\
		0 \ar[r] & T(M(B)\times A) \ar@{{ |>}->}[r] \ar@<.5ex>@{.>}[u] & T(A\cupdot B) \ar@<.5ex>[u]^-{\psi} \ar@{-{ >>}}[r]_-{g} & T(B) \ar@<.5ex>[u]^-{T(f|_B)} \ar[r] & 0
		}
	\]
	This proves that $\Ker(T(f))$ is a projective object, as a retract of the free Lie algebra $T(M(B)\times A)$. The result now follows from Proposition~\ref{propo_equivalent_conditions}.
\end{proof}

\section{Chain complexes and homology}\label{Section-Chains-Homology-Resolutions}

We return to the main subject of our article and start working towards a non-additive approach to derived functors. For the sake of clarity, we recall some straightforward non-abelian versions of well-established basic definitions of (abelian) homological algebra---see, for instance,~\cite{VdLinden:Doc} or~\cite{JuliaThesis} for further details.

\subsection{Graded objects}
Let $\C$ be any category. A \defn{graded object} in $\C$ is a sequence of objects $(C_{n})_{n\in \ZZ}$ or, equivalently, a functor $C\colon{\ZZ\to \C}$ where the set of integers~$\ZZ$ is considered as a discrete category. Consider $k\in\ZZ$. A \defn{morphism $f\colon {A\to B}$ of degree~$k$} between graded objects~$A$ and~$B$ is a collection of morphisms $(f_{n}\colon A_{n}\to B_{n+k})_{n \in \ZZ}$ in $\C$. That is to say, $f$ is a natural transformation from $A$ to $B s^{k}$ where the map~$s^{k}\colon\ZZ\to \ZZ$ sends $n$ to $n+k$. Of course, a morphism $f\colon {A\to B}$ of degree $k$ composes with a morphism $g\colon {B\to C}$ of degree $l$ to a morphism $g f\colon{A\to C}$ of degree $k+l$.

\subsection{Chain complexes}\label{subsec : chain complexes}
Now let $\C$ be pointed, and let $0$ denote the constant graded object~$(0)_{n\in \ZZ}$. We shall also write $0$ for any morphism (of whichever degree) that factors over $0$.

A \defn{chain complex} $(C,d)$ in $\C$ is a graded object $C$ together with a morphism $d \colon{C\to C}$ of degree $-1$ (its \defn{differential}) such that~${d  d =0}$. So~$d$ is a sequence
\[
	\xymatrix{\cdots \ar[r] & C_{n+1} \ar[r]^-{d_{n+1}} & C_{n} \ar[r]^-{d_{n}} & C_{n-1} \ar[r] & \cdots}
\]
in which the composite of any two successive morphisms is trivial. A chain complex~$(C,d)$ is \defn{bounded below} when there exists $i\in \ZZ$ such that $C_{n}=0$ for all~${n< i}$, and \defn{positively graded} when this $i$ is $0$. Sometimes we omit the differential $d$ in our notation and write $C$ for a chain complex $(C,d)$.

A \defn{chain morphism} $f\colon{(C,d^{C})\to (E,d^{E})}$ from a chain complex $(C,d^{C})$ to a chain complex $(E,d^{E})$ is a morphism $f\colon{C\to E}$ of degree $0$ such that $d^{E} f=f d^{C}$. Chain complexes in $\C$ and chain morphisms between them form a category which we write~$\Ch (\C)$.

Any zero-preserving functor $F \colon{\C \to\E}$ between pointed categories $\C$ and $\E$ induces a functor $F \colon{\Ch(\C) \to \Ch(\E)}$ between the respective categories of chain complexes.

\subsection{Exact complexes, proper complexes and homology}\label{subsec - exact complexes, proper complexes and homology}
We assume that $\C$ is pointed with kernels and cokernels.
We say that a chain complex~$(C,d)$ is \defn{exact at $C_{n}$} or \defn{exact in degree $n$} when the lifting~$\bar{d}_{n+1}\colon {C_{n+1} \to \Z_n(C)}$ of $d_{n+1}$ over the kernel $\Z_n(C)$ of $d_{n}$ is a regular epimorphism.
\[
	\xymatrix@!0@=4em{\cdots \ar[r] & C_{n+1} \ar[rr]^-{d_{n+1}} \ar@{-->}[rd]_-{\bar{d}_{n+1}} & & C_{n} \ar[rr]^-{d_{n}} & & C_{n-1} \ar[r] & \cdots \\
	& & \Z_n(C) \ar@{{ |>}->}[ru]_-{\kker(d_n)} & & }
\]
The complex $C$ is \defn{exact} when it is exact in all degrees: the lifting $\bar{d}$ of $d$ over $\kker (d)\colon{\Z(d)\to C}$ is a regular epimorphism. In the context of a normal category, this is equivalent to saying that it consists of short exact sequences spliced together.

A chain complex~$(C,d)$ is \defn{proper} when $d$ admits a factorization as a regular epimorphism $e$ followed by a normal monomorphism $m$.
\[
	\xymatrix@!0@=4em{\cdots \ar[r] & C_{n+1} \ar[rr]^-{d_{n+1}} \ar@{-{>>}}[rd]_-{e_{n+1}} & & C_{n} \ar[rr]^-{d_{n}} & & C_{n-1} \ar[r] & \cdots \\
	& & \Ima(d_{n+1}) \ar@{{ |>}->}[ru]_-{m_{n+1}} & & }
\]
This means that each $d_{n+1}$ is a proper morphism. Clearly, every exact complex is proper.

Given a proper chain complex~$(C,d)$ in a normal category $\C$ and $n\in \ZZ$, the \defn{$n$th homology object $\H_{n}(C)$ of~$C$} is the cokernel of $\bar{d}_{n+1}\colon {C_{n+1} \to \Z_n(C)}$.
\begin{equation}\label{Definition of homology diagram}
	\vcenter{\xymatrix@!0@R=4em@C=4em{\cdots \ar[r] & C_{n+1} \ar[rr]^-{d_{n+1}} \ar[rd]_-{\bar{d}_{n+1}} & & C_{n} \ar[rd]|-{\bar{d}_{n}} \ar[rr]^-{d_{n}} & & C_{n-1} \ar[r] & \cdots \\
	& & \Z_n(C) \ar@{-{ >>}}[rd]_-{\coker(\bar{d}_{n+1})} \ar@{{ |>}->}[ru]|-{\kker(d_n)} & & \Z_{n-1}(C) \ar@{{ |>}->}[ru]_-{\kker(d_{n-1})} \ar@{-{ >>}}[rd]^-{\coker(\bar{d}_{n})} & \\
	& & & \H_n(C) \ar@{-->}[ru]  & & \H_{n-1}(C) & }}
\end{equation}
All homology objects together form a graded object $\H (C)=\Coker(\bar{d})$, in fact, a chain complex with zero differentials. For each $n\in \ZZ$, the operations $\Z_n$ and $\H_n$ define functors from $\Ch(\C)$ to $\C$ such that if $f \colon{C \to E}$ is chain morphism, then
\[
	\kker(d^E_n) \Z_n(f) = f_n \kker(d^C_n)
\]
and
\[
	\H_n(f) \coker(\bar{d}^C_{n+1})  = \coker(\bar{d}^E_{n+1}) \Z_n(f) .
\]

In a normal category, proper chain complexes are precisely those in which homology may detect exactness, so that a proper complex $C$ is exact if and only if~${\H (C)=0}$; see~\cite{VdLinden:Doc} for further details.

\section{A non-additive Horseshoe Lemma}\label{sec - non add Horseshoe lemma}

\subsection{Projective resolutions}
As recalled in Section~\ref{sec - derived functor ab case}, a particular type of chain complex which plays a central role in the construction of the additive left derived functors are the projective resolutions of an object. We first consider this notion in the context of a pointed category with kernels, then specialize to pointed regular categories and homological categories.

\begin{figure}
	\resizebox{\textwidth}{!}{
	\xymatrix@R=2em@C=2em{
	\cdots\ar[r] & C_{n+1}\ar[rr]^-{d_{n+1}}\ar@{-{>>}}[dr]_-{\bar{d}_{n+1}} & & C_n \ar[r]^-{d_{n}} & C_{n-1}\ar[r] & \cdots \ar[r] & C_1 \ar@{-{>>}}[dr]_-{\bar{d}_{1}} \ar[rr]^-{d_1} && C_0\ar@{-{>>}}[rd]^-{d_0} \ar[r] & 0 \\
	&  & \Ker(d_n) \ar@{{ |>}->}[ur]_-{\ker(d_n)} & & & & & \Ker(d_0) \ar@{{ |>}->}[ur]_-{\ker(d_0)} && X
	}}
	\caption{A projective resolution of $X$}\label{Figure Projective Resolution}
\end{figure}

\begin{definition}\label{def projective resolution}
	Consider an object $X$ in a pointed category with kernels. A \defn{(regular) projective resolution of $X$} consists of a positively graded chain complex $(C,d)$ as in Figure~\ref{Figure Projective Resolution}, all of whose objects $C_i$ are projective, together with a regular epimorphism $d_0\colon C_0\to X$ such that for each $n\geq 0$, the lifting $\bar{d}_{n+1}\colon {C_{n+1}\to \Ker(d_n)}$ of $d_{n+1}$ over the kernel $\Ker(d_n)$ of $d_n$ is a regular epimorphism.
\end{definition}

As in the abelian setting, the existence of projective resolutions is guaranteed when enough projective objects are available; see for instance~\cite{Cartan-Eilenberg} for a full proof. The idea is to take a projective cover $d_0\colon C_0\to X$ of $X$: the object $C_0$ is the first object of the resolution. We then take the kernel $\ker(d_0)\colon\Ker(d_0)\to C_0$ of~$d_0$, and cover this object with a projective object $C_1$. The induced composite is $d_1\colon C_1\to C_0$. The procedure continues by recursion, and we find:

\begin{proposition}\label{prop enough proje ensure proje resol}
	A pointed category with kernels has enough projectives if and only if each object admits a projective resolution.\noproof
\end{proposition}

For this reason, we will assume that our categories have enough projectives. Since this implies regularity of the category as soon as it has finite limits and coequalizers of kernel pairs, we will assume this condition as well.

\begin{remark}[Trivial resolution for a projective object]\label{rem - projective resol when X proj}
	If $X$ is projective, then it admits the chain complex $C$ where $C_0=X$ and $C_n=0$ for all $n\neq 0$ as a projective resolution.
\end{remark}

\begin{example}[Projective resolutions in varieties of algebras]
	We consider a pointed variety $\V$, and as in Proposition~\ref{propo_equivalent_conditions} denote the counit of the adjunction $T\dashv U$ by $\varepsilon$. Since its components are regular epimorphisms, any object admits a projective resolution.

	In the case of a pointed Schreier variety, we can find a ``short'' projective resolution which is trivial in degrees greater than $1$. Indeed, as mentioned in Example~\ref{ex_schreier variety}, in such a category, any projective object is a free object, and any subobject of a free object is again free. So if $X$ is an object and $d_0\colon C_0\to X$ a surjective algebra morphism where $C_0$ is free object, then its kernel $C_1$ is again free, and the construction stops: we may take all other $C_n$ equal to zero.
\end{example}

\begin{figure}
	%\resizebox{\textwidth}{!}
	{\(
	\xymatrix@C=3.5em{\cdots \ar[r] & C_1 \ar@{-->}[d]_-{f_1} \ar[r]^-{d^C_1} & C_0 \ar@{-->}[d]_-{f_0} \ar@{-{>>}}[r]^-{d_0^C} & X \ar[d]^-{x}\\
	\cdots \ar[r] & E_1 \ar[r]_-{d^E_1} & E_0 \ar@{-{>>}}[r]_-{d_0^E} & Y}
	\)}\caption{A lifting of a morphism $x$}\label{Chain morphism over f}
\end{figure}

In order to define the left derived functors of a functor, we first of all need that any morphism $x\colon X\to Y$ in the domain category induces a morphism of projective resolutions $f\colon {C(X)\to E(Y)}$ \defn{over $x$}, also called a \defn{lifting of $x$}, which means that $x\comp d_0^C=d_0^E\comp f_0$ as in Figure~\ref{Chain morphism over f}. This is possible thanks to the next result, of which the proof is straightforward, as in the abelian context---see for instance~\cite{MacLane:Homology}.

\begin{lemma}\label{Lemma Existence Lifting}
	In a pointed category with kernels, suppose projective resolutions $C(X)$ and $E(Y)$ of objects $X$ and~$Y$ are given. Then any morphism $x\colon {X\to Y}$ lifts to a morphism of projective resolutions $f\colon {C(X)\to E(Y)}$.\noproof
\end{lemma}

Note that an object $X$ may admit several non-isomorphic projective resolutions. However, thanks to the above, we know that between them, we can always find a morphism of projective resolutions over $1_X$. In the abelian case, such an $f$ happens to be a chain homotopy equivalence. In order to recover this result, we need an appropriate notion of \emph{chain homotopy}, valid in a non-additive context. This will be the subject of Section~\ref{sec - approx homotopy}.

In a normal category, the concept of a projective resolution considered above may be characterized in terms of homology objects as in the abelian case.

\begin{remark}\label{remark projective resolution}
	Consider an object $X$ in a normal category. A chain complex $C$
	\[
		\xymatrix{
		\cdots\ar[r] & C_{n+1}\ar[r]^-{d_{n+1}} & C_n \ar[r] & \cdots \ar[r] & C_1 \ar[r]^-{d_1} & C_0 \ar[r] & 0 \ar[r]& \cdots
		}
	\]
	is a projective resolution of $X$ if it is projective in all degrees, while $\H_0(C)=X$ and $\H_n(C)=0$ for all $n\neq 0$. Here the morphism $d_0\colon C_0\to X$ is nothing but $\coker(d_1)\colon C_0\to X=\H_0(C)$. This is the viewpoint of Figure~\ref{Figure Projective Resolution Intro}.
\end{remark}

We now restrict our attention to the context of a homological category, where a version of the Horseshoe Lemma is available.

\begin{figure}
	\(
	\xymatrix@C=3.5em{
	& & & & & 0\ar[d] & \\
	& & & & & X \ar@{{ |>}->}[d]^-{\alpha_{-1}}  & \\
	& & & & & Y \ar@{-{ >>}}[d]^-{\beta_{-1}} & \\
	\cdots \ar[r]& E_{n+1}\ar[r]_-{d^E_{n+1}}  & E_{n}\ar[r] & \cdots\ar[r]  &  E_0 \ar@{-{ >>}}[r]_-{\coker(d_1^E)} &  Z\ar[r] \ar[d] & 0\\
	& & & & & 0 &
	}
	\)
	\caption{The diagram to be filled with resolutions of $X$ and of $Y$}\label{fig - Setting for horseshoe}
\end{figure}

\subsection{The Horseshoe Lemma in homological categories}

The remainder of the present section is dedicated to obtaining a non-additive version of the so-called \emph{Horseshoe Lemma}, which is essential in the construction of the long exact sequence of derived functors (Section~\ref{sec - syzygy}, \cite{MacLane:Homology, Cartan-Eilenberg}). We will see that the availability of this result is related to Condition \P. Given a diagram such as the one in Figure~\ref{fig - Setting for horseshoe} where the vertical sequence is exact and the bottom exact sequence represents a projective resolution $E$ such that $\H_0(E)=Z$, we would like to find projective resolutions of $X$ and $Y$ and morphisms such as in Figure~\ref{fig - Filled horseshoe} where all vertical sequences are exact. In a homological category, this is indeed possible---as soon as Condition \P\ holds; note that since in this context, biproducts are not available, the classical abelian proof strategy cannot work here.

\begin{proposition}[Half Horseshoe Lemma in homological categories]\label{proposition:horseshoe lemma semiab}
	In a homological category with enough projectives which satisfies Condition \P, consider a diagram as in Figure \ref{fig - Setting for horseshoe} where the vertical sequence is exact and the horizontal sequence is a projective resolution of $Z$. Then there exist projective resolutions of $X$ and~$Y$, respectively, which complete the diagram in Figure \ref{fig - Filled horseshoe} depicting a vertical short exact sequence of horizontal exact sequences.
\end{proposition}

We say ``half'' because unlike the abelian case we only start with an arbitrary resolution of $Z$, not of $Z$ and $X$, as one does in the abelian case.

\begin{figure}[t]
	\(
	\xymatrix@C=3.5em{
	& 0 \ar[d] & 0 \ar[d] & & 0 \ar[d] & 0 \ar[d] & \\
	\cdots \ar[r] & C_{n+1} \ar[r]^ -{d_{n+1}^C} \ar@{{ |>}->}[d]_ -{\alpha_{n+1}}  & C_n \ar[r] \ar@{{ |>}->}[d]^ -{\alpha_{n}}  & \cdots \ar[r] & C_0 \ar@{-{ >>}}[r]^-{\coker(d_1^C)} \ar@{{ |>}->}[d]_ -{\alpha_{0}}  & X \ar@{{ |>}->}[d]^ -{\alpha_{-1}} \ar[r] & 0\\
	\cdots \ar[r] & A_{n+1}\ar[r]^ -{d_{n+1}^A} \ar@{-{ >>}}[d]_-{\beta_{n+1}} & A_n\ar[r] \ar@{-{ >>}}[d]^ -{\beta_{n}} & \cdots \ar[r] & A_0 \ar@{-{ >>}}[r]^ - {\coker(d_1^A)} \ar@{-{ >>}}[d]_ -{\beta_{0}} & Y \ar@{-{ >>}}[d]^ -{\beta_{-1}} \ar[r] & 0 \\
	\cdots\ar[r] & E_{n+1}\ar[r]_-{d^E_{n+1}} \ar[d] & E_{n}\ar[r] \ar[d] & \cdots\ar[r] & E_0 \ar@{-{ >>}}[r]_ -{\coker(d_1^E)} \ar[d] &  Z \ar[r]\ar[d] & 0\\
	& 0 & 0 & & 0 & 0 &
	}
	\)
	\caption{The diagram of Figure~\ref{fig - Setting for horseshoe} filled}\label{fig - Filled horseshoe}
\end{figure}
\begin{figure}[t]
	\(
	\xymatrix@C=3.5em{
	0\ar[d]& 0\ar[d] & 0 \ar[d] & \\
	K_{\beta_0} \ar@{{ |>}->}[d]_-{\kker (\beta_0)}  \ar@{-->}[r]^-{\overline{\varepsilon_0}}&K_{p_0}\ar[r]^-{\cong} \ar@{{ |>}->}[d]_-{\kker (p_0)} & X\ar[r] \ar@{{ |>}->}[d]^ -{\alpha_{-1}} & 0\\
	P_0\ar@{ >>}[r]^-{\varepsilon_0} \ar@{-{ >>}}@<.5ex>[d]^-{\beta_0} &P \halfsplitpullback \ar@{ >>}[r]^-{p_1} \ar@{-{ >>}}@<.5ex>[d]^-{p_0}   & Y \ar[r] \ar@{-{ >>}}[d]^ -{\beta_{-1}} & 0\\
	E_0\ar[d] \ar@{{ >}->}@<.5ex>[u]^-{i_0} \ar@{=}[r] &E_0 \ar@{{ >}->}@<.5ex>[u]^-{i} \ar[d] \ar@{-{ >>}}[r]_-{\coker(d_1^E)} & Z\ar[r] \ar[d] & 0 \\
	0&0 & 0 &
	}
	\)
	\caption{Take the pullback $P$ and a projective cover of it}\label{fig - Starting diagram}
\end{figure}
\begin{proof}
	The proof is by recursion on $n\in \NN$. Let us start with the first step, $n=0$.

	We consider Figure~\ref{fig - Starting diagram} where the bottom square is a pullback. Being pullbacks of normal epimorphisms, $p_0$ and $p_1$ are normal epimorphisms as well. The morphism between the kernels $K_{p_0}$ and $X$ is an isomorphism by~\cite[Lemma 1]{Bourn2001}. Notice that the normal epimorphism $p_0$ is split, by some section $i$, since $E_0$ is a projective object.	Since by assumption the category has enough projectives, $P$ is the quotient, via a normal epimorphism $\varepsilon_0$, of a projective object $P_0$. The composite $\beta_0=p_0\comp \varepsilon_0$, as a normal epimorphism with codomain $E_0$, is a split epimorphism with a section we denote $i_0$. Taking the kernel of this split epimorphism $\beta_0$ will show us that $X$ is a quotient of a projective object, namely $K_{\beta_0}$ via $\overline{\varepsilon_0}$.

	Since the vertical sequence on the left is a split short exact sequence where $P_0$ is a projective object, Condition \P\ implies that the kernel object $K_{\beta_0}$ is a projective object as well. Finally, since the top left square is a pullback, $\varepsilon_0$ being a normal epimorphism implies that the induced morphism $\overline{\varepsilon_0}$ is a normal epimorphism as well. If we set $C_0=K_{\beta_0}$, then $X$ is a quotient of $K_{\beta_0}$ via $\eta_0^C=\overline{\varepsilon_0}$ and if we set $A_0=P_0$, then $Y$ is a quotient of $P_0$ via $\eta_0^A=p_1\varepsilon_0$.

	\begin{figure}[t]
		\(
		\xymatrix@C=4em{
		&  & 0\ar[d]& & 0\ar[d] &  \\
		0\ar[r] & K_{\eta_n^C}\ar@{-->}[d]_-{\overline{\alpha_n}} \ar@{{ |>}->}[r]_-{\kker (\eta_n^C)}   & C_n \ar@/^1pc/[rr]^-{d^C_n} \ar@{ >>}[r]_-{\eta_n^C}  \ar@{{ |>}->}[d]_-{\alpha_n}  & K_{\eta_{n-1}^C}\ar@{{ |>}->}[r]_-{\kker (\eta^C_{n-1})}   & C_{n-1} \ar@{{ |>}->}[d]_-{\alpha_{n-1}}  \ar[r] & \cdots  \\
		0\ar[r] & K_{\eta_n^A} \ar@{-->}[d]_-{\overline{\beta_n}} \ar@{{ |>}->}[r]_-{\kker (\eta_n^A)}   & A_n \ar@{-{ >>}}[d]_ -{\beta_{n}} \ar@{ >>}[r]_-{\eta_n^A} \ar@/^1pc/[rr]^-{d^A_n}  & K_{\eta_{n-1}^A} \ar@{{ |>}->}[r]_-{\kker (\eta^A_{n-1})}  & A_{n-1} \ar@{-{ >>}}[d]^ -{\beta_{n-1}} \ar[r] & \cdots  \\
		0\ar[r] & K_{\eta_n^E} \ar@{{ |>}->}[r]_-{\kker (\eta_n^E)}   & E_n \ar@{ >>}[r]_-{\eta_n^E} \ar@/^1pc/[rr]^-{d^E_n} \ar[d]  & K_{\eta_{n-1}^E} \ar@{{ |>}->}[r]_-{\kker (\eta^E_{n-1})}  & E_{n-1}\ar[d]\ar[r] & \cdots  \\
		&  & 0 & & 0 &
		}
		\)
		\caption{Construction of $d_n^C$ and $d_n^D$}\label{Partial_n^C and partial_n^D}
	\end{figure}
	\begin{figure}[b]
		\(
		\xymatrix@C=3.5em{
		&  & 0\ar[d]&  0\ar[d] &  \\
		0\ar[r] & K_{\eta_n^C}\ar@{-->}[d]_-{\overline{\alpha_n}} \ar@{{ |>}->}[r]_-{\kker (\eta_n^C)}   & C_n  \ar@{ >>}[r]_-{\eta_n^C}  \ar@{{ |>}->}[d]_-{\alpha_n}  & K_{\eta_{n-1}^C} \ar@{{ |>}->}[d]_-{\overline{\alpha_{n-1}}} \ar[r] & 0  \\
		0\ar[r] & K_{\eta_n^A} \ar@{-->}[d]_-{\overline{\beta_n}} \ar@{{ |>}->}[r]_-{\kker (\eta_n^A)}   & A_n \ar@{-{ >>}}[d]_ -{\beta_{n}} \ar@{ >>}[r]_-{\eta_n^A}  & K_{\eta_{n-1}^A} \ar@{-{ >>}}[d]_ -{\overline{\beta_{n-1}}}   \ar[r] & 0  \\
		0\ar[r] & K_{\eta_n^E} \ar@{{ |>}->}[r]_-{\kker (\eta_n^E)}   & E_n \ar@{ >>}[r]_-{\eta_n^E}  \ar[d]  & K_{\eta_{n-1}^E} \ar[d] \ar[r] & 0  \\
		&  & 0 &  0 &
		}
		\)
		\caption{Part of Figure~\ref{Partial_n^C and partial_n^D} on which the $(3\times 3)$-Lemma is applied}\label{Figure recursion step}
	\end{figure}
	For the recursion step, let us consider Figure~\ref{Partial_n^C and partial_n^D} where we suppose that partial projective resolutions $X$ and $Y$ are defined respectively until the projective objects $C_n$ and $A_n$. On the one hand, notice that we defined $d_n^C= \kker ( \eta_{n-1}^C) \eta_n^C $ where $\eta_n^C$ is a normal epimorphism by assumption. We defined $d_n^A$ in an analogous way. On the other hand, we have $d_n^E=\kker (\eta_{n-1}^E) \eta_n ^ E$ by exactness of the projective resolution $E$ of $Z$ at the level $n-1$. Hence, the sequences are projective resolutions up to level $n-1$.

	In Figure~\ref{Figure recursion step} shows a $(3\times 3)$-diagram where we have three horizontal short exact sequences (by definition) and two vertical short exact sequences on the right (by assumption). By the $(3\times 3)$-Lemma \cite[Theorem 12]{Bourn2001}, the vertical sequence with the dashed arrow is a short exact sequence as well. We return to the situation of the the base step, with $K_{\eta_n^C}$, $K_{\eta_n^A}$ and $K_{\eta_n^E}$ playing, respectively, the roles of $X$, $Y$ and $Z$, and the normal epimorphism $\eta_{n+1}^E \colon E_{n+1}\to K_{\eta_{n}^E}$ playing the role of $\coker(d_1^E)$.
\end{proof}

\begin{remark}[Degreewise split short exact sequence]
	Like in the abelian context, the previous proof leads us to a short exact sequence of chain complexes
	\begin{equation}\label{Horseshoe SES}
		\xymatrix{0 \ar[r] & C(X) \ar@{{ |>}->}[r] & A(Y) \ar@{-{ >>}}[r] & E(Z) \ar[r]  & 0}
	\end{equation}
	where for each $n\geq 0$ we have a split short exact sequence
	$$\xymatrix{0 \ar[r] & C_n \ar@{{ |>}->}[r] & A_n \ar@{-{ >>}}@<.5ex>[r] & E_n \ar[r] \ar@{{ >}->}@<.5ex>[l] & 0.}$$
	However, in the current context, $A_n$ need not be a product of $C_n$ and $E_n$.
\end{remark}

\begin{remark}\label{Remark Horseshoe Lemma SSES}
	Suppose the vertical sequence in Figure~\ref{fig - Setting for horseshoe} is \emph{split} exact, and that a splitting $\gamma_{-1}\colon Z\to Y$ is given. Then the morphism $i$ in Figure~\ref{fig - Starting diagram} can be chosen to be induced by $\gamma_{-1}$ in such a way that $\gamma_{-1}\coker(d_1^E)=p_1i$. Now $i_0$ is any lifting of $i$ over $\varepsilon_0$. Repeating this procedure, we see that the sequence~\eqref{Horseshoe SES} is split exact. In other words, the procedure of Proposition~\ref{proposition:horseshoe lemma semiab}, applied to a split short exact sequence in $\C$, yields a split short exact sequence of chain complexes.
\end{remark}

Condition \P\ plays a key role in the previous proof. Actually, it is not only sufficient, but also necessary:

\begin{proposition}
	If, in a homological category with enough projectives $\C$, whenever we consider the diagram of Figure~\ref{fig - Setting for horseshoe} as in Proposition~\ref{proposition:horseshoe lemma semiab}, it can be completed into a short exact sequence of resolutions as in the conclusion of that result, then Condition \P\ holds for $\C$.
\end{proposition}
\begin{proof}
	Consider a split short exact sequence
	\[
		\xymatrix{0 \ar[r] & K \ar@{{ |>}->}[r]^-k & X \ar@{-{ >>}}@<.5ex>[r]^-{f} & Y \ar[r] \ar@{{ >}->}@<.5ex>[l]^-{s} & 0}
	\]
	where $X$ is a projective object. Then $Y$ is a projective object as well. We prove that $K$ is also projective.

	Since $\C$ has enough projectives, by Proposition~\ref{prop enough proje ensure proje resol} we may consider a projective resolution of $Y$. By assumption we can construct suitable projective resolutions of $K$ and $X$; in particular, we obtain the diagram
	\[
		\xymatrix{
		0 \ar[r] & C_0 \ar@{-{ >>}}[d]_ - {\coker(d_1^C)}  \ar@{{ |>}->}[r]^-{\alpha_0}  & A_0 \ar@{-{ >>}}[d]|-{\coker(d_1^A)}  \ar@<.5ex>@{-{ >>}}[r]^-{\beta_0} & E_0 \ar@{{ >}->}@<.5ex>[l]^-{j} \ar@{-{ >>}}[d]^ - {\coker(d_1^E)} \ar[r] & 0 \\
		0 \ar[r] & K \ar@{{ |>}->}[r]^-k & X \ar@{-{ >>}}@<.5ex>[r]^-{f} & Y \ar[r] \ar@{{ >}->}@<.5ex>[l]^-{s} & 0}
	\]
	whose horizontal sequences are short exact. Since $X$ and $Y$ are projective objects, the normal epimorphisms $\coker(d_1^A)$ and $\coker(d_1^E)$ are split. As a result, and since a homological category is always regular Mal'tsev~\cite{Borceux-Bourn}, we are in a similar situation as in the proof for (ii) $\To$ (i) of Proposition~\ref{propo_equivalent_conditions}. We may thus choose suitable sections of $\coker(d_1^A)$ and $\coker(d_1^E)$ which restrict to a section of $\coker(d_1^C)$. This shows that $K$ is a retract of the projective object $C_0$. Hence $K$ is projective as well, and \P\ holds in the category $\C$.
\end{proof}

\section{Approximate differences}\label{sec - approximate difference}

As mentioned in Subsection~\ref{sec - explanation non additive context}, the non-additive context does not allow us to properly define an addition between parallel arrows---certainly not one which forms a group structure. However, we can define a notion of \emph{approximate difference}, which may still be available in a non-additive context. This section explains what we need to make this work: it is inspired by the theory of so-called \emph{approximate subtractions}~\cite{DB-ZJ-2009b,DB-ZJ-2009}. Originally developed in the context of \emph{subtractive categories}~\cite{ZJanelidze-Subtractive, ZJanelidze-Snake}, this theory allows the use of (in some suitable sense \emph{imaginary} or \emph{approximate}) differences even when the hom-sets have no abelian group structure on them.

Observe that the difference of two morphisms $f$, $g\colon{X\to Y}$ in an abelian category may be rewritten as
\begin{equation}\label{Additive}
	\vcenter{\xymatrix@!0@R=4em@C=6em{0 \ar[r] & X \ar@{{ |>}->}[r]^-{\univ{{1_X}}{{-1_X}}} \ar[rd]_-{f-g} & X\oplus X \ar@{-{ >>}}@<.5ex>[r]^-{\nabla_X} \ar[d]^-{\couniv{f}{g}} & X \ar@{{ >}->}@<.5ex>[l]^-{\iota_1} \ar[r] & 0\\
	&& Y}}
\end{equation}
where the horizontal sequence is exact and where $f-g = \couniv{f}{g} \univ{ 1_X}{ -1_X }$ (see for instance \cite{Freyd} for further explanation). We can mimic this procedure in the context of a pointed category: given any object $X$ for which the needed (co)limits exist, we set $(D(X),\delta_{X})$ be the kernel of $\nabla_X=\couniv{1_{X}}{1_{X}}\colon{X+X\to X}$ and call it the \defn{difference object} of $X$.
\begin{equation*}\label{Additive_non_ab}
	\vcenter{\xymatrix@!0@R=4em@C=6em{D(X) \ar@{{ |>}->}[r]^-{\delta_X} \ar[rd]_-{f-g} & X+ X \ar@<.5ex>[r]^-{\nabla_X} \ar[d]^-{\couniv{f}{g}} & X \ar@{{ >}->}@<.5ex>[l]^-{\iota_1} \\
	& Y}}
\end{equation*}
If now $f$, $g\colon{X\to Y}$, then we view the composite $\couniv{f}{g}\comp \delta_X$ as a kind of \emph{approximate difference} of $f$ and $g$, and write it $f-g\colon D(X)\to Y$. Note that even though this definition needs very little from the category where it is made, whether it is useful depends strongly on the category's properties. For our purposes, it will be convenient to work in a homological category with coproducts, which encompasses the subtractive setting of~\cite{DB-ZJ-2009}.

\begin{example}[Groups] \label{ex - D(Grp)}
	For any group $X$, the kernel $D(X)$ is a subgroup of the free group $X+X$. We can prove that it is generated by elements of the shape
	\[
		\overline{x}\underline{x}^{-1} \quad\text{and}\quad \underline{y}\overline{y}^{-1}
	\]
	for $x$, $y\in X$ and where we write $\overline{(\cdot)}$ and $\underline{(\cdot)}$ for the elements which belong, respectively, to the first or second copy of $X$ inside the coproduct $X+X$.
\end{example}

\begin{example}[Lie algebras]\label{ex - D(LieK)}
	In the category $\LieK$, it is easy to see that for two algebras $X$ and $Y$, the coproduct $X+Y$ is the set of (equivalence classes, with respect to anticommutativity and the Jacobi identity, as well as the relations that hold in $X$ and in $Y$) of all (non-commutative, non-associative) polynomials with variables in $X$ and $Y$ (see for instance~\cite[Section 10]{VdL:nonassociativealgebra} for a more detailed explanation).

	The elements of $D(X)$ ``are'' polynomials generated by monomials of three distinct types: a first one, coming from the vector space structure, where we have $\overline{z}-\underline{z}$ and $\underline{z}-\overline{z}$ for $z$ any (non-associative) word with letters in $X$; a second, coming from the anticommutativity of the Lie bracket, where we have $\overline{z}\underline{z}$ and $\underline{z}\overline{z}$; and a third type induced by the Jacobi identity, where we find monomials of the shape $z_1(z_2 z_3)+z_2(z_3 z_1)+z_3(z_1 z_2)$, for words with letters in $X$---here we need to  exclude the situation where the $z_i$ belong to the same copy of $X$.
\end{example}

\subsection{Calculus of subtraction}\label{sec - calculus of subtraction}
For now, we assume that $\C$ is a pointed category with functorially chosen kernels and binary coproducts. This allows us to define the \defn{difference object functor} $D\colon \C\to \C$. Given a morphism $h\colon {X \to Y}$, we let~$D(h)$ be the unique induced arrow in the diagram
\begin{equation}\label{eq - D(h)}
	\vcenter{\xymatrix@!0@=5em{D(X) \ar@{{ |>}->}[r]^-{\delta_{X}} \ar@{-->}[d]_-{D(h)} & X+X \ar@<.5ex>[r]^-{\nabla_X} \ar[d]_-{h+h} & X \ar[d]^-{h} \ar@{{ >}->}@<.5ex>[l]^-{\iota_1}\\
	D(Y) \ar@{{ |>}->}[r]_-{\delta_{Y}} & Y+Y \ar@<.5ex>[r]^-{\nabla_Y} & Y \ar@{{ >}->}@<.5ex>[l]^-{\iota_1}}}
\end{equation}
We will also need the natural transformation $\varsigma\colon D \To 1_{\C}$ defined for $X$ in $\C$ by
\[
	\varsigma_X \defeq \couniv{1_{X}}{0} \delta_{X}\colon D(X)\to X.
\]
To see that $\varsigma$ is indeed natural, so that $f\varsigma_X=\varsigma_Y D(f)$ for any arrow $f\colon X\to Y$, it suffices to take a look at the diagram
\[
	\xymatrix{
	D(X)\ar@{{ |>}->}[r]^-{\delta_{X}} \ar[dd]_-{D(f)} \ar[rrd] _-(0.3){\varsigma_X} & X+X\ar[rr]^-(0.75){\couniv{1_X}{1_X}}\ar[dd]|(0.28)\hole_-{f+f}\ar[rd]|-{\couniv{1_X}{0}} & & X\ar[dd]^-{f}\\
	& & X\ar[dd]_-(0.3){f} & \\
	D(Y)\ar@{{ |>}->}[r]^-{\delta_{Y}} \ar[rrd]_-(0.3){\varsigma_Y} & Y+Y\ar[rr]|(0.55)\hole_-(0.75){\couniv{1_Y}{1_Y}} \ar[rd]|-{\couniv{1_Y}{0}} && Y \\
	& & Y &
	}
\]

\begin{lemma}\label{lemma : D preserves 0}
	Under the above assumptions on $\C$, the functor $D\colon \C\to \C$ preserves the zero morphism.\noproof
\end{lemma}

Approximate differences are compatible with composition:

\begin{lemma}\label{lemma: distributivity for minus}
	Under the above assumptions on $\C$, we have
	\[
		l (f-g)=lf-lg\qquad \text{and} \qquad (f-g)D(h)=fh-gh
	\]
	whenever these equalities make sense.
\end{lemma}
\begin{proof}
	The second equality is a consequence of the commutativity of~\eqref{eq - D(h)} while the first follows immediately from the universal property of the coproduct.
\end{proof}

\begin{remark}\label{rem:eqsub}
	For future use, we take note of the equality $f-0= f\varsigma_{\mathrm{dom}(f)}$ which holds for any morphism $f$.
\end{remark}

For the moment it is not so clear whether the definition of $f-g$ encodes a ``difference'' as in the abelian context, in the sense that $f=g$ if and only if $f-g=0$. Under certain conditions on the surrounding category, this does indeed happen.

\begin{lemma}\label{lemma : f= g ssi f-g=0}
	Under the above assumptions on $\C$, we have
	\[
		f=g\quad\To\quad f-g=0.
	\]
	If, in addition, $\C$ is a normal category, then \(f-g=0\) implies \(f=g\).
\end{lemma}
\begin{proof}
	For the first implication, as mentioned in~\cite[Section 3.1]{DB-ZJ-2009b}, this follows from the fact that $\couniv{f}{f}=f\couniv{1_X}{1_X}$ while $\delta_X$ is the kernel of $\couniv{1_X}{1_X}$, so that
	\[
		f-g=\couniv{f}{f} \delta_X=f\couniv{1_X}{1_X}\delta_X=0.
	\]

	If $\C$ is a normal category, then as already mentioned in~\cite[Introduction]{DB-ZJ-2009b}, the codiagonal $\couniv{1_X}{1_X}$ is the cokernel of $\delta_X$. By assumption, we have that $\couniv{f}{g} \delta_X=0$. Hence $\couniv{f}{g}$ factors through $\couniv{1_X}{1_X}$ via a morphism $h$ so that $\couniv{f}{g}=h\couniv{1_X}{1_X}=\couniv{h}{h}$ and $f=h=g$.
\end{proof}

As a consequence, normality can be characterized in terms of approximate differences.

\begin{proposition}\label{Proposition Real Differences equivalent to Zero}
	A pointed regular finitely cocomplete category is normal if and only if for any two given morphisms~$f$,~$g\colon{X\to Y}$,
	\[
		f-g = 0\quad\To\quad f=g.
	\]
\end{proposition}
\begin{proof}
	Lemma~\ref{lemma : f= g ssi f-g=0} shows that the implication holds under the given assumptions. For the converse, it suffices to prove that the coequalizer of $f$, $g\colon X\to Y$ is a cokernel of $f-g\colon D(X)\to Y$. This implies that all regular epimorphisms are normal. If indeed $h(f-g)=hf-hg=0$, then $hf=hg$ by assumption, so that the universal property of the coequalizer implies the universal property of the cokernel.
\end{proof}

\subsection{Subtractive categories}\label{sec:Subtractive Cats}
Next to the above property, we will also need that
\[
	f-0=g-0\qquad\To\qquad f=g.
\]
Since $f\varsigma_X=f-0=g-0=g\varsigma_X$ by Remark~\ref{rem:eqsub}, this happens as soon as $\varsigma_X$ is an epimorphism. We ask slightly more, because then we find ourselves in a known situation: the context of \emph{subtractive categories}. We will not use the original definition of subtractivity here, but rather an equivalent formulation~\cite[Theorem 5.1]{DB-ZJ-2009}, valid for pointed regular categories with binary coproducts:

\begin{definition}\label{Definition Subtractive Category}
	A pointed regular category with binary coproducts is called \defn{subtractive} when for each object $X$, the composite
	\[
		1_X - 0 =\couniv{1_{X}}{0}\delta_{X}=\varsigma_X \colon{D(X)\to X}
	\]
	is a regular epimorphism.
\end{definition}

We see that the natural transformation $\varsigma$ is a (component-wise) regular epimorphism if and only if we have a subtractive category. Via this short observation it is easy to see that any abelian category is subtractive. Indeed, in this setting, $\delta_X=\univ{ 1_X}{ -1_X}$, and the domain of the difference is $D(X)=X$.

Almost by definition, subtractive categories form a context in which a difference of two morphisms may be expressed as an approximate difference. The disadvantage of Definition~\ref{Definition Subtractive Category} is that it does not immediately show the relationships with other categorical-algebraic conditions. Subtractive categories were introduced in~\cite{ZJanelidze-Subtractive} as a categorical characterization of subtractive varieties of universal algebras, studied earlier in~\cite{Ursini3}, in a context which is not necessarily regular and where binary coproducts need not exist. A relation $r=\univ{ r_{1}}{r_{2}}\colon {R\to X\times Y}$ in a pointed category is said to be \defn{left (right) punctual}~\cite{Bourn2002} if $\univ{1_X}{0}\colon {X\to X\times Y}$ (respectively $\univ{0}{1_Y}\colon {Y\to X\times Y}$) factors through $r$. A pointed finitely complete category~$\C$ is said to be \defn{subtractive}, if every left punctual reflexive relation on any object $X$ in~$\C$ is right punctual.

By~\cite[Theorem 5.1]{DB-ZJ-2009}, in a pointed regular category with binary coproducts, this agrees with Definition~\ref{Definition Subtractive Category}. The arguments in~\cite{ZJanelidze-Subtractive} further show that any pointed protomodular category is subtractive, so that any additive category with finite limits, any {homological category} in the sense of~\cite{Borceux-Bourn}, and in particular any {semi-abelian category}~\cite{Janelidze-Marki-Tholen} is subtractive.

It is explained in~\cite{ZJanelidze-Subtractive} that a variety of algebras is subtractive when the corresponding theory contains a unique constant $0$ and a binary operation $-$, called a \defn{subtraction}, subject to the equations $x-0 = x$ and $x-x = 0$. So, for instance, any pointed variety of algebras whose theory contains a group operation is an example.

\subsection{A stronger context for subtraction}\label{sec -homological cat good for D and sigma}
Since we are eventually taking homology in the codomain of a functor $F\colon \C \to \E$, we shall be working with minimal assumptions on $\C$ (a pointed regular category with binary coproducts, with enough projectives that satisfies Condition \P) and use the stronger but natural context of a \emph{homological} category for $\E$. Homological categories are a suitable framework for our calculus of differences for four reasons:
\begin{enumerate}
	\item they admit kernels and cokernels of proper morphisms;
	\item they are subtractive;
	\item they are normal;
	\item the functor $D\colon \E\to\E$ preserves normal epimorphisms.
\end{enumerate}

Our aim is now to prove that the fourth point does indeed hold. Here we again use that a homological category is always regular Mal'tsev. Recall~\cite{EGoeVdL,Bourn2003,Carboni-Kelly-Pedicchio} that a regular category is \defn{Mal'tsev} if and only if any split epimorphism of regular epimorphisms is a regular pushout---see the proof of Proposition~\ref{propo_equivalent_conditions} which recalls what this means.

\begin{lemma}\label{Lemma D preserves regular epi}
	If $\C$ is a pointed regular Mal'tsev category with binary coproducts, then the functor $D\colon \C\to \C$ preserves regular epimorphisms.
\end{lemma}
\begin{proof}
	If $h$ is a regular epimorphism, then $h+h$ is a regular epimorphism as well. Hence if $\C$ is regular Mal'tsev, then in the diagram~\eqref{eq - D(h)} defining $D(h)$, the square on the right is a regular pushout.
	Let $P$ denote the pullback of $\nabla_Y$ and $h$, and write $h^*(\nabla_Y)\colon {P\to X}$ for the induced arrow. It is then easy to see that on the one hand, the kernel of $h^*(\nabla_Y)$ is $D(Y)$, while, on the other hand, the square
	\[
		\xymatrix@!0@=5em{D(X) \pullback \ar@{->}[r]^-{\delta_{X}} \ar@{->}[d]_-{D(h)} & X+X \ar@{-{>>}}[d]^-{\univ{h+h}{\nabla_X}} \\
		D(Y) \ar@{->}[r] & P }
	\]
	is a pullback. It follows that the arrow $D(h)$ is a regular epimorphism.
\end{proof}

\begin{remark}
	The functor $D$ need not be exact, not even in the context of a semi-abelian category. We consider the alternating group $A_3$ as a normal subgroup of the symmetric group~$S_3$. In the notation of Example~\ref{ex - D(Grp)}, we have that for $\overline{(12)}\underline{(12)}\in D(S_3)$ and $\underline{(123)}\overline{(132)}\in D(A_3)$ the conjugate
	\[
		\overline{(12)}\underline{(12)}\underline{(123)}\overline{(132)}\underline{(12)}\overline{(12)}=\overline{(12)}\underline{(13)}\overline{(132)}\underline{(12)}\overline{(12)}
	\]
	in $D(S_3)$ is not in $A_3+A_3$ and therefore not in $D(A_3)$.
\end{remark}

\subsection{Differences and the functor $D$}

Even though the functor $D$ is engaged in the construction of the approximate difference, outside of the additive context it need not interact well with those differences. It is indeed not true that $D(f-g)$ is equal to $D(f)-D(g)$ in general. If this were the case, then by Lemma~\ref{lemma : D preserves 0} we should have $D(\varsigma_X)=D(1_X-0)=D(1_X)-D(0)=\varsigma_{D(X)}$.
However, in $\Gp$,
\[
	D(\varsigma_X)\bigl(\overline{(\underline{x}\overline{x}^{-1})} \underline{(\overline{x}\underline{x}^{-1})} \bigr)=\overline{x}^{-1}\underline{x}\neq \underline{x}\overline{x}^{-1}=\varsigma_{D(X)}\bigl(\overline{(\underline{x}\overline{x}^{-1})}\underline{(\overline{x}\underline{x}^{-1})} \bigr)
\]
holds for all $x\in X$ by Example~\ref{ex - D(Grp)}.

\begin{remark}[Compositions of $\varsigma_X$]\label{rem : composition of varsigma_X}
	We know that in general, $\varsigma_{D(X)}\neq D(\varsigma_X)$. However, thanks to the naturality of $\varsigma$, for any object $X$ we have
	\[
		\varsigma_X \varsigma_{D(X)}=\varsigma_X D(\varsigma_X)\colon D^2(X)\to X.
	\]
	We will also write this composition $\varsigma_X^2$ since its domain is $D^2(X)$. By induction, for all $n\geq 2$ we find the equality
	\begin{equation}
		\varsigma_X \varsigma_{D(X)}\cdots \varsigma_{D^{n-1}(X)}=\varsigma_X D(\varsigma_X)\cdots D^{n-1}(\varsigma_X)\colon D^n(X)\to X \label{eq :equivalence def compos sigma}
	\end{equation}
	and we will denote this morphism $\varsigma_X^n$.

	By convention, we will extend the notation $\varsigma_X^n$ to all natural numbers by letting $\varsigma_X^0=1_X$ (since its domain will be $D^0(X)=X$) and $\varsigma_X^1=\varsigma_X$ (since its domain will be $D(X)$). Note that in a subtractive category, all the morphisms $\varsigma_X^n$ are regular epimorphisms (as composites of such).

	For any $f\colon X\to Y$ and all $n\in \NN$, we have
	\begin{equation}\label{eq - compatibility sigma's and D's}
		f \varsigma^n_X=\varsigma^n_Y D^n(f)
	\end{equation}
	so that the $\varsigma^n_X$ form a natural transformation $\varsigma^n\colon D^n\To 1_\C$.
\end{remark}

\section{Approximate homotopies}\label{sec - approx homotopy}
\begin{figure}
	%\resizebox{\textwidth}{!}
	{\(
	\xymatrix{\cdots \ar[r] & C_2 \ar@{.>}[ld] \ar@<-.5ex>[d]_-{f_2} \ar@<.5ex>[d]^-{g_2} \ar[r]^-{d^C_2} &C_1 \ar@{.>}[ld]|-{h_1} \ar@<-.5ex>[d]_-{f_1} \ar@<.5ex>[d]^-{g_1} \ar[r]^-{d^C_1} & C_0 \ar@{.>}[ld]|-{h_0} \ar@<-.5ex>[d]_-{f_0} \ar@<.5ex>[d]^-{g_0} \ar[r] & 0 \ar@{.>}[dl]^-(.4){h_{-1}=0}\\
	\cdots \ar[r] & E_2 \ar[r]_-{d^E_2} & E_1 \ar[r]_-{d^E_1} & E_0 \ar[r] & 0}
	\)}
	\caption{The lowest degrees of a chain homotopy}\label{Fig Chain}
\end{figure}
Recall that given a pair $f$, $g\colon{(C,d^C)\to (E,d^E)}$ of parallel morphisms of positively graded chain complexes in an abelian category---see Figure~\ref{Fig Chain}---a \defn{chain homotopy} from $f$ to $g$ is a morphism of graded objects $h\colon{(C,d^C)\to (E,d^E)}$ of degree $1$ such that
\[
	f - g = d^{E}\comp h + h\comp d^{C}
\]
holds.

\begin{figure}[h]
	%\resizebox{\textwidth}{!}
	{\(
	\xymatrix@C=3.5em{\cdots \ar[r] & C_1 \ar@<-.5ex>[d]_-{f_1} \ar@<.5ex>[d]^-{g_1} \ar[r]^-{d^C_1} & C_0 \ar@<-.5ex>[d]_-{f_0} \ar@<.5ex>[d]^-{g_0} \ar@{-{ >>}}[r]^-{\coker(d_1^C)} & X \ar[d]^-{x}\\
	\cdots \ar[r] & E_1 \ar[r]_-{d^E_1} & E_0 \ar@{-{ >>}}[r]_-{\coker(d_1^E)} & Y}
	\)}\caption{Two liftings of a morphism $x$}\label{Setting for the construction of h0}
\end{figure}
In an abelian category, the lifting of a morphism $x\colon X\to Y$ to a morphism of resolutions of $X$ and $Y$ given in Lemma~\ref{Lemma Existence Lifting} is unique up to a chain homotopy (see for instance~\cite{MacLane:Homology} for a full proof). Given two such liftings as in Figure~\ref{Setting for the construction of h0}, a homotopy is obtained through a recursive construction, which depends on projectivity of the objects of $C$ and the exactness of the bottom sequence---any projective resolution having both features. In the first step, for instance, $h_0$ is obtained as a lifting over the epimorphism $\bar{d}_1^E$ of the lifting of $f_0-g_0$ through the image of $d_1^E$. That lifting does indeed exist, because by the exactness of the bottom sequence, the image of $d_1^E$ is the kernel of its cokernel, while
\[
	\coker(d_1^E)\comp(f_0-g_0)=(x-x)\comp \coker(d_1^C)=0.
\]
This argument uses that the domain of a difference of two morphisms with a projective domain (in this case, the domain of $f_0-g_0$) is still a projective object, which is true in all abelian categories, because there the domain of a difference of two parallel morphisms is just the domain of those morphisms themselves. In a semi-abelian category, this is not automatic. For the difference object $D(X)$ to be projective as soon as~$X$ is, we are led to assuming that the category satisfies Condition~\P.

\subsection{Towards a definition}\label{subsec - definition by hand of approximate homotopy}
We let $\C$ be a normal category with binary coproducts and enough projectives which satisfies Condition \P. Let $C(X)$ and~$E(Y)$ be projective resolutions of, respectively, $X$ and $Y$, two objects of $\C$. Given a morphism $x\colon X\to Y$, we know from Lemma~\ref{Lemma Existence Lifting} that a morphism of projective resolutions $f\colon {C(X)\to E(Y)}$ exists such that $\H_0(f)=x$. Let us now assume that $g\colon {C(X)\to E(Y)}$ is another such morphism (Figure~\ref{Setting for the construction of h0} again). We follow the example of the abelian context, to recursively construct what should amount to a chain homotopy from $f$ to $g$. Since the standard definition does not immediately make sense here, we shall simply proceed with the construction itself and take what comes out as a starting point for a non-additive definition.

\begin{figure}[h]
	%\resizebox{\textwidth}{!}
	{\(\xymatrix@C=3.5em{
	& & D(C_0) \ar@/_/@{-->}[ddl]|-(0.3){\alpha_0} \ar@/_/@{-->}[ddll]_-(0.35){h_0} \ar@{->}@/_4ex/[dd]^-(.2){f_0-g_0}& \\
	C_1\ar[rr]^-(0.3){d^C_1} \ar@<-.5ex>[d]_-(.3){f_1} \ar@<.5ex>[d]^-(.3){g_1} &  & C_0 \ar@<-.5ex>[d]_-(0.3){f_0} \ar@<.5ex>[d]^-(0.3){g_0} \ar@{-{ >>}}[r]^-{\coker(d_1^C)} & X \ar[d]^-{x} \\
	E_1 \ar@{ >>}[r]_-{\bar{d}_1^E } & \Ima(d_1^E) \ar@{{ |>}->}[r]_-{\kker(\eta^E)} & E_0 \ar@{-{ >>}}[r]_-{\coker(d_1^E)} & Y}
	\)}\caption{Construction of $h_0$}\label{Constructing h_0 Intro}
\end{figure}
\subsubsection*{Construction of $h_0$}
From Lemma~\ref{lemma : f= g ssi f-g=0} and Lemma~\ref{lemma: distributivity for minus}, it follows that the composite $\coker(d_1^E)\comp(f_0-g_0)$ is zero. Hence $f_0-g_0$ lifts over the image of $d_1^E$ via a morphism $\alpha_0$ as in Figure~\ref{Constructing h_0 Intro}. By the assumption that the class of projective objects is closed under protosplit subobjects, the difference object $D(C_0)$ is projective. Hence, since the bottom chain complex is exact at each level, there exists a lifting $h_0$ of $\alpha_0$ over $\bar{d}_1^E$. The construction ensures that
\[
	d_1^E h_0=f_0-g_0.
\]
We remark that this equality is the same in the abelian context. The only novelty here is the domain $D(C_0)$ shared by $h_0$ and the difference $f_0-g_0$.

\begin{figure}
	%\resizebox{\textwidth}{!}
	{\(\xymatrix@C=3.5em{
	& & D^2(C_1) \ar@/_/@{-->}[ddl]|-(0.3){a_1} \ar@/_/@{-->}[ddll]_-(0.35){h_1} \ar@{->}@/_4ex/[dd]^-(.2){(f_1-g_1)-h_0D(d_1^C)}& \\
	C_2\ar[rr]^-(0.3){d^C_2} \ar@<-.5ex>[d]_-(.3){f_2} \ar@<.5ex>[d]^-(.3){g_2} &  & C_1 \ar@<-.5ex>[d]_-(0.3){f_1} \ar@<.5ex>[d]^-(0.3){g_1} \ar[r]^-{d_1^C} & C_0  \ar@<-.5ex>[d]_-{f_0} \ar@<.5ex>[d]^-{g_0} \\
	E_2 \ar@{ >>}[r]_-{\bar{d}_2^E} & \Ima(d_2^E) \ar@{{ |>}->}[r]_-{\kker(d_1^E)} & E_1 \ar[r]_-{d_1^E} & E_0}
	\)}\caption{Construction of $h_1$}\label{Constructing h_1 Intro}
\end{figure}
\subsubsection*{Construction of $h_1$}
The construction of $h_1$, the second morphism in the homotopy, depends on the same strategy and the same preliminary results about differences of morphisms: see Figure~\ref{Constructing h_1 Intro}. Since
\begin{align*}
	d_1^E((f_1-g_1)-h_0 D(d_1^C)) & = d_1^E(f_1-g_1)-d_1^Eh_0 D(d_1^C)           \\
	                              & = (f_0-g_0)D(d_1^C) - (f_0-g_0)D(d_1^C) = 0,
\end{align*}
there exists a unique morphism $\alpha_1$ such that
\[
	\kker(d_1^E) \alpha_1 = (f_1-g_1) - h_0 D(d_1^C).
\]
Using the exactness of the bottom chain complex and the fact that $D^2(C_1)$ is projective, we obtain a lifting $h_1$ of $\alpha_1$ over $\bar{d}_2^E$. We find
\[
	d_2^E h_1 = (f_1-g_1)-h_0 D(d_1 ^C).
\]

\subsubsection*{Construction of $h_2$}
We repeat the same strategy once again: we have
\begin{align*}
	 & d_2^E \left((f_2-g_2)\varsigma_{D(C_2)}-h_1D^2(d_2^C) \right)                                                                              \\
	 & \overset{\ref{lemma: distributivity for minus}}{\qeq} d_2^E(f_2-g_2)\varsigma_{D(C_2)}-d_2^Eh_1D^2(d_2^C)                                  \\
	 & \overset{\text{def.\ $h_{1}$}}{\qeq} d_2^E(f_2-g_2)\varsigma_{D(C_2)}-\left( (f_1-g_1)-h_0D(d_1^C)\right) D^2(d_2^C)                       \\
	 & \overset{\ref{lemma: distributivity for minus}}{\qeq} d_2^E(f_2-g_2)\varsigma_{D(C_2)}-\left( (f_1-g_1)D(d_2^C)-h_0D(d_1^C)D(d_2^C)\right) \\
	 & \equal d_2^E(f_2-g_2)\varsigma_{D(C_2)}-\left( (f_1-g_1)D(d_2^C)-h_0D(d_1^C d_2^C)\right)                                                  \\
	 & \equal d_2^E(f_2-g_2)\varsigma_{D(C_2)}-\left( (f_1-g_1)D(d_2^C)-h_0D(0)\right)                                                            \\
	 & \overset{\ref{lemma : D preserves 0}}{\qeq} d_2^E(f_2-g_2)\varsigma_{D(C_2)}-\left( (f_1-g_1)D(d_2^C)-0\right)                             \\
	 & \overset{\ref{rem:eqsub}}{\qeq} d_2^E(f_2-g_2)\varsigma_{D(C_2)}-(f_1-g_1)D(d_2^C)\left( 1_{D(C_2)}-0\right)                               \\
	 & \overset{\ref{rem:eqsub}}{\qeq} d_2^E(f_2-g_2)\varsigma_{D(C_2)}-(f_1-g_1)D(d_2^C)\varsigma_{D(C_2)}                                       \\
	 & \overset{\ref{lemma: distributivity for minus}}{\qeq} d_2^E(f_2-g_2)\varsigma_{D(C_2)}-(f_1d_2^C-g_1d_2^C)\varsigma_{D(C_2)}               \\
	 & \equal d_2^E(f_2-g_2)\varsigma_{D(C_2)}-(d_2^E f_2-d_2^Eg_2)\varsigma_{D(C_2)}                                                             \\
	 & \overset{\ref{lemma: distributivity for minus}}{\qeq} d_2^E(f_2-g_2)\varsigma_{D(C_2)}-d_2^E(f_2-g_2)\varsigma_{D(C_2)}                    \\
	 & \overset{\ref{lemma : f= g ssi f-g=0}}{\qeq}0.
\end{align*}
Therefore, by the exactness of the bottom chain complex and by the fact that $D^3(C_2)$ is projective, we obtain a lifting $h_3$ of the unique morphism induced by the kernel of $d_2^E$ over $\bar{d}_3^E$. In summary: $d_3^E h_2 = (f_2-g_2)\varsigma_{D(C_2)}-h_1 D^2(d_2^C)$.

\subsubsection*{General case}
The construction of $h_2$ already gives us all the tools for defining~$h_n$ when $n\geq 3$. Following essentially the same strategy (which now involves Equation~\eqref{eq - compatibility sigma's and D's} as well) we find a lifting $h_n$ that satisfies
\begin{equation*}\label{eq - def of hn}
	d_{n+1}^E h_n = (f_n-g_n)\varsigma^{n-1}_{D(C_n)} - h_{n-1} D^n (d _n ^C).
\end{equation*}

By the convention explained in Remark~\ref{rem : composition of varsigma_X}, this general formula agrees with the cases $n=1$ and $n=2$. However, the case $n=0$ differs slightly.
In summary: with the previous assumptions and notation, we have a collection of morphisms $(h_n)_{n\in \NN}$ where $h_n\colon D^{n+1}(C_n) \to E_{n+1}$ such that
\begin{equation}\label{eq - def of hn for n=0 and general}
	\begin{cases}
		d_1^E h_0 = (f_0-g_0)                                                  &                       \\
		d_{n+1}^E h_n= (f_n-g_n)\varsigma^{n-1}_{D(C_n)} - h_{n-1} D^n (d_n^C) & \text{for $n\geq 1$}.
	\end{cases}
\end{equation}

In the context of a pointed category with kernels and binary coproducts, we can do exactly the same thing, but the notation is slightly different ($d_0$ instead of $\coker(d_1)$, see Remark~\ref{remark projective resolution}). This leads us to the following definition.

\begin{definition}\label{Def Approx Chain Homotopy}
	Let $\C$ be a pointed category with kernels and binary coproducts. Consider morphisms $f$, $g\colon C\to E$ between positively graded chain complexes $C$ and $E$. An \defn{(approximate) (chain) homotopy} from $f$ to $g$ is a sequence of morphisms $(h_n\colon D^{n+1}(C_n)\to E_{n+1})_{n\in \NN}$ for which Equations \eqref{eq - def of hn for n=0 and general} hold. We write $h\colon f\simeq g$.

	Two positively graded chain complexes $C$ and $E$ are \defn{(approximately) chain homotopically equivalent} ($C\simeq D$) when chain morphisms $\phi \colon C\to E$ and $\psi \colon E\to C$ exist, together with approximate homotopies $1_C\simeq\psi \phi$ and $1_E\simeq\phi \psi$.
\end{definition}

We proved:

\begin{theorem}[Homotopy unique existence of projective resolutions]\label{thm - projective resolution unique up to approx homoto}
	Let $\C$ be a pointed category with kernels and binary coproducts for which Condition $\P$ holds. Any lifting of a morphism ${x\colon X\to Y}$ in the sense of Lemma~\ref{Lemma Existence Lifting} is unique up to an approximate chain homotopy.

	When, moreover, $\C$ has enough projectives, then for any object $X\in \C$ there exists a projective resolution, unique up to an approximate chain homotopy equivalence.\noproof
\end{theorem}

This result is of key importance in ensuring that the non-additive left-derived functors we consider in Section~\ref{Section Derived Functors} are well defined.

\begin{remark}[Extension to bounded below chain complexes]\label{rem - extension to bounded below}
	By renumbering, Definition~\ref{Def Approx Chain Homotopy} extends to bounded below chain complexes.
\end{remark}

\begin{remark}[Comparison to the abelian case]
	The equalities in \eqref{eq - def of hn for n=0 and general} are expressed in a non-additive context. In an abelian category, we recover the usual: a collection of morphisms $(h_n\colon C_n \to E_{n+1})_{n\in \NN}$ such that
	\begin{equation*}
		\begin{cases}
			d_1^E h_0 = (f_0-g_0)                    &                       \\
			d_{n+1}^E h_n= (f_n-g_n) - h_{n-1} d_n^C & \text{for $n\geq 1$}.
		\end{cases}
	\end{equation*}
\end{remark}

\subsection{Basic stability properties}
First of all, it is easy to see that the approximate homotopy relation is compatible with pre- and post-composition with chain morphisms. Indeed, consider a chain homotopy $h\colon f\simeq g\colon C\to E$. Lemma~\ref{lemma: distributivity for minus} implies that if $\alpha\colon A\to C$ is a chain morphism, then $(h_n D^{n+1}(\alpha_n))_n$ is an approximate homotopy from $f\alpha$ to $g\alpha$, while if $\beta\colon E\to B$ is a chain morphism, then $(\beta_{n+1}h_n)_n$ is an approximate homotopy from $\beta f$ to $\beta g$.

Lemma~\ref{lemma : f= g ssi f-g=0} implies that $0\colon f\simeq f$ for any chain morphism $f$, while $0\colon f\simeq g$ entails $f=g$ in any normal subtractive category. In particular, the homotopy relation is always reflexive. We next prove that is also symmetric. Here we need the following construction.

For any object $X$, we consider the \defn{twist morphism} $\tw_{D(X)}\colon D(X)\to D(X)$, the restriction of $\couniv{\iota_2}{\iota_1}$ to $D(X)$:
\[
	\xymatrix@!0@R=4em@C=6em{
	D(X)\ar[d]_-{\tw_{D(X)}} \ar@{{ |>}->}[r]^-{\delta_X} & X+ X \ar[r]^-{\nabla_X} \ar[d]^-{\couniv{\iota_2}{\iota_1}} & X  \ar@{=}[d] \\
	D(X) \ar@{{ |>}->}[r]_-{\delta_X}& X+X \ar[r]_-{\nabla_X}& X
	}
\]
It is natural, in the sense that for any morphism $l\colon X\to Y$,
\[
	D(l)\tw_{D(X)}=\tw_{D(Y)} D(l).
\]
Moreover, for any arrows $f$, $g\colon X\to Y$, the equation
\[
	(f-g)\tw_{D(X)}=\couniv{f}{g}\delta_X \tw_{D(X)} = \couniv{f}{g}\couniv{\iota_2}{\iota_1}\delta_X = \couniv{g}{f}\delta_X = g-f
\]
holds.

\begin{lemma}[Symmetry of the approximate homotopy relation]\label{lemma - symmetric relation}
	In any pointed category with kernels and binary coproducts, $f\simeq g$ implies $g \simeq f$.
\end{lemma}
\begin{proof}
	Consider $h\colon f\simeq g$. We construct an approximate homotopy from $g$ to $f$ by recursion. For the base step, we let $h'_0\coloneq h_0 \tw_{D(C_0)}$ so that
	\[
		d_1^Eh'_0  =(f_0-g_0)\tw_{D(C_0)}=g_0-f_0.
	\]
	For $n\geq 1$, we assume that $h'_{n-1}\coloneq h_{n-1} D^{n-1}(\tw_{D(C_{n-1})})$ is defined and satisfies
	\[
		d_{n}^E h_{n-1}'= (g_n-f_n)\varsigma^{n-1}_{D(C_n)} - h'_{n-2} D^{n-1}(d_{n-1}^C).
	\]
	If now $h'_n\coloneq h_nD^n(\tw_{D(C_n)})$, then
	\begin{align*}
		d_{n+1}^E h'_n & = \bigl( (f_n-g_n)\varsigma^{n-1}_{D(C_n)} - h_{n-1} D^n (d_n^C) \bigr)D^n(\tw_{D(C_n)})             \\
		               & = (f_n-g_n)\varsigma^{n-1}_{D(C_n)} D^{n-1}(\tw_{D(C_n)})- h_{n-1} D^n (d_n^C) D^{n-1}(\tw_{D(C_n)}) \\
		               & = (f_n-g_n) \tw_{D(C_n)} \varsigma^{n-1}_{D(C_n)}- h_{n-1} D^{n-1}\bigl( D(d_n^C) \tw_{D(C_n)}\bigr) \\
		               & =(g_n-f_n)\varsigma^{n-1}_{D(C_n)}- h_{n-1} D^{n-1}\left(\tw_{D(C_{n-1})}D(d_n^C)\right)             \\
		               & =(g_n-f_n)\varsigma^{n-1}_{D(C_n)}- h_{n-1} D^{n-1}(\tw_{D(C_{n-1})})D^n(d_n^C)                      \\
		               & =(g_n-f_n)\varsigma^{n-1}_{D(C_n)}- h_{n-1}'D^n(d_n^C).\qedhere
	\end{align*}
\end{proof}

\begin{remark}[Symmetry in the abelian case]
	It is well known that if $h\colon f\simeq g$, then $-h\colon g\simeq f$ in any abelian category. This agrees with what we do in the above proof, because here the morphism $\tw_{D(X)}$ is simply $-1_X$. To see this, it suffices to recall that $\delta_X=\univ{ 1_X}{-1_X}$.
\end{remark}

\begin{figure}[h]
	%\resizebox{\textwidth}{!}
	{\(
	\xymatrix@C=3pc @R=3pc{
	D(\Z_n(C)) \ar@/_2.5pc/[dd]_-{\Z_n(f)-\Z_n(g)} \ar@{-->}[d]|-{k_n^C} \ar[dr]|-{D(\kker (d_n^C))} \ar@/^/[drr]^- {0}& & \\
	\Z_n(D(C))\ar[d]|-{\Z_n(f-g)} \ar@{{ |>}->}[r]_-{\kker(D(d_n^C))} & D(C_n) \ar[r]_-{D(d_n^C)} \ar[d]|-{f_n-g_n}& D(C_{n+1})\ar[d]|-{f_{n-1}-g_{n-1}}\\
	\Z_n(E) \ar@{{ |>}->}[r]_-{\kker(d_n^E)}& E_n \ar[r]_-{d_n^E} & E_{n+1}}
	\)}
	\caption{Comparison between $\Z_n(f)-\Z_n(g)$ and $\Z_n (f-g)$} \label{fig - Z vs minus}
\end{figure}
\subsection{Approximate homotopy and homology}
If, in an abelian category, two chain morphisms $f$, $g\colon (C,d^C)\to (E,d^E)$ are homotopic, then $\H_n(f)=\H_n(g)$ for all $n$ (see for instance~\cite{Cartan-Eilenberg}). As recalled in Section~\ref{sec - derived functor ab case}, the construction of the left-derived functors crucially depends on this fact.

The usual (abelian) proof simplifies because the homology functors are additive. This is no longer true in a non-additive setting. The reason is that the kernel functor need not preserve differences. More precisely, in a pointed finitely complete category with binary coproducts, $\Z_n(f-g)$ and $\Z_n(f)-\Z_n(g)$ need not agree. As depicted in Figure~\ref{fig - Z vs minus}, there is a canonical comparison $k_n^C$, the unique morphism for which $D(\kker (d_n^C))= \kker (D(d_n^C)) k_n^C$. Hence
\[
	\Z_n(f-g)k_n^C  = \Z_n(f)-\Z_n(g).
\]
This morphism $k_n^C$ need not be an isomorphism in general. For this reason, the non-additive proof of the compatibility of chain homotopy with homology is slightly more involved.

\begin{theorem}\label{thm - approx homotopy and homology}
	If, in a normal subtractive category, $f\simeq g$, then $\H_n(f)=\H_n(g)$ for all $n\in \NN$.
\end{theorem}
\begin{proof}\label{proof : homotopy vs homology}
	In a normal subtractive category $\C$, suppose that $h$ is a chain homotopy $f\simeq g\colon (C,d^C)\to (E,d^E)$. Consider $n\geq 1$. We observe that, by definition of $\H_n$, it suffices to show
	\[
		\H_n(f) \coker(\bar{d}^C_{n+1})= \H_n(g) \coker(\bar{d}^C_{n+1}).
	\]
	By the construction of the functors $\Z_n$ and $\H_n$ as in Diagram~\eqref{Definition of homology diagram}, this is equivalent to
	\[
		\coker(\bar{d}^E_{n+1}) \Z_n(f) = \coker(\bar{d}^E_{n+1}) \Z_n(g).
	\]
	Since $\C$ is normal, by Lemma~\ref{lemma : f= g ssi f-g=0} and Lemma~\ref{lemma: distributivity for minus}, it suffices that
	\[
		\coker(\bar{d}^E_{n+1}) \Z_n(f) - \coker(\bar{d}^E_{n+1}) \Z_n(g) = \coker(\bar{d}^E_{n+1}) \ (\Z_n(f)- \Z_n(g)) =  0.
	\]
	Since $\C$ is subtractive and by definition of $\varsigma^n$ (see Remark~\ref{rem : composition of varsigma_X}), the equality
	\[
		\coker(\bar{d}^E_{n+1}) \ (\Z_n(f)- \Z_n(g)) \varsigma^n_{D(\Z_n(C))} =  0
	\]
	is sufficient for this to happen.
	% since $\sigma^n_{D(\Z_n(C))}$ is a regular epimorphism.
	\begin{figure}
		%\resizebox{\textwidth}{!}
		{
		\(\xymatrix{
		C_{n+1} \ar@<.5ex>[rr]^-{f_{n+1}} \ar@<-1 ex>[rr]_-{g_{n+1}}  \ar[dd]_-{d_{n+1}^C}  \ar[dr]|-{\bar{d}_{n+1}^C} & & E_{n+1} \ar[dd]|(0.48)\hole|(0.56)\hole_-(0.2){d_{n+1}^E}  \ar[dr]|-{\bar{d}_{n+1}^E}   & D^{n+1}(C_n) \ar[l]_-{h_n} &   & D^{n+1}(\Z_n(C)) \ar@{ >>}[d]^-{\varsigma^n_{D(\Z_n(C))}} \ar[ll]_-{D^{n+1}(\kker (d_n ^C))} \\
		& \Z_{n}(C) \ar@<.5ex>[rr]^-(0.3){\Z_n(f)}  \ar@{{ |>}->}[dl]|-{\kker (d^C_n)} \ar@{-{ >>}}[dd]^-(0.7){\coker (\bar{d}_{n+1}^C)} \ar@<-1 ex>[rr]_-(0.3){\Z_n(g)}  & & \Z_n(E) \ar@{{ |>}->}[dl]|-{\kker (d^E_n)}  \ar@{-{ >>}}[dd]^-(0.7){\coker (\bar{d}_{n+1}^E)}  & & D(\Z_n (C)) \ar[ll]^-{\Z_n(f)-\Z_n(g)}\\
		C_n \ar@<.5ex>[rr]|\hole^-(0.3){f_{n}} \ar@<-1 ex>[rr]|\hole_-(0.3){g_{n}}  & & E_n & & &\\
		& \H_{n}(C)\ar@<.5ex>[rr]^-{\H_n(f)}  \ar@<-1 ex>[rr]_-{\H_{n}(g)}  & & \H_n(E) & & } \)
		}
		\caption{Homotopy vs.\ homology}\label{Fig Hom vs Hom}
	\end{figure}
	We prove that $(\Z_n(f)- \Z_n(g)) \varsigma^n_{D(\Z_n(C))} $ lifts over $\bar{d}_{n+1}^E$ as in Figure~\ref{Fig Hom vs Hom}:
	\begin{align*}
		 & \kker (d_n^E) \bar{d}_{n+1}^Eh_n D^{n+1}(\kker (d_n ^C))                                                                                           \\
		 & \equal d_n^Eh_n D^{n+1}(\kker (d_n ^C))                                                                                                            \\
		 & \overset{\text{def.\ $h_{n}$}}{\qeq} ((f_n-g_n)\varsigma^{n-1}_{D(C_n)} - h_{n-1}D^n(d_n^C))D^{n+1}(\kker (d_n ^C))                                \\
		 & \overset{\ref{lemma: distributivity for minus}}{\qeq} (f_n-g_n)\varsigma^{n-1}_{D(C_n)} D^n(\kker (d_n ^C)) - h_{n-1}D^n(d_n^C)D^n(\kker (d_n ^C)) \\
		 & \equal (f_n-g_n)\varsigma^{n-1}_{D(C_n)} D^n(\kker (d_n ^C)) - h_{n-1}D^n(d_n^C \kker (d_n ^C))                                                    \\
		 & \equal (f_n-g_n)\varsigma^{n-1}_{D(C_n)} D^n(\kker (d_n ^C)) - h_{n-1}D^n(0)                                                                       \\
		 & \overset{\ref{lemma : D preserves 0}}{\qeq} (f_n-g_n)\varsigma^{n-1}_{D(C_n)} D^n(\kker (d_n ^C)) - 0                                              \\
		 & \overset{\ref{rem:eqsub}}{\qeq} (f_n-g_n)\varsigma^{n-1}_{D(C_n)} D^n(\kker (d_n ^C)) (1_{D^n(\Z_n(C))} - 0)                                       \\
		 & \equal (f_n-g_n)\varsigma^{n-1}_{D(C_n)} D^n(\kker (d_n ^C)) \varsigma_{D^n(\Z_n(C))}                                                              \\
		 & \overset{\eqref{eq - compatibility sigma's and D's}}{\qeq} (f_n-g_n)D(\kker (d_n ^C)) \varsigma^{n-1}_{D(\Z_n(C))} \varsigma_{D^n(\Z_n(C))}        \\
		 & \overset{\ref{rem : composition of varsigma_X}}{\qeq} (f_n-g_n)D(\kker (d_n ^C)) \varsigma^{n}_{D(\Z_n(C))}                                        \\
		 & \overset{\ref{lemma: distributivity for minus}}{\qeq} (f_n \kker (d_n ^C)-g_n \kker (d_n ^C)) \varsigma^{n}_{D(\Z_n(C))}                           \\
		 & \overset{\text{def.\ $\Z_{n}$}}{\qeq} (\kker (d_n ^E)\Z_n(f)-\kker (d_n ^E)\Z_n(g)) \varsigma^{n}_{D(\Z_n(C))}                                     \\
		 & \overset{\ref{lemma: distributivity for minus}}{\qeq}\kker (d_n ^E) (\Z_n(f)-\Z_n(g)) \varsigma^{n}_{D(\Z_n(C))}.
	\end{align*}
	Since $\kker (d_n ^E)$ is monic, this implies that
	\[
		\bar{d}_{n+1}^Eh_n D^{n+1}(\kker (d_n ^C)) = (\Z_n(f)-\Z_n(g)) \varsigma^{n}_{D(\Z_n(C))}.
	\]

	To complete the proof, we still have to verify the case $n=0$. Here we can use essentially the same strategy as above. We have:
	\begin{align*}
		\kker(\eta^E)\bar{d}_{1}^Eh_0D(\kker (\eta^C)) & =d_1^Eh_0D(\kker (\eta^C))               \\
		                                               & = (f_0-g_0)D(\kker (\eta^C))             \\
		                                               & =f_0 \kker (\eta^C) - g_0 \kker (\eta^C) \\
		                                               & =\kker (\eta^E) (\Z_0(f)-\Z_0(g)).
	\end{align*}
	Since $\kker (\eta^E)$ is monic, we have $\bar{d}_{1}^E h_0 D(\kker(\eta^C)) = \Z_0(f)-\Z_0(g)$, which proves that ${\Z_0(f)-\Z_0(g)}$ lifts over $\bar{d}_{1}^E$.
\end{proof}

\section{Functors preserving approximate homotopies}\label{sec - discussion functors}
In the abelian context, \emph{additive} functors (see Section~\ref{sec - derived functor ab case}) are the ones which admit derived functors, essentially because being compatible with the addition in the hom-sets, they preserve chain homotopies. The aim of this section is to find convenient conditions on functors so that they preserve \emph{approximate} homotopies. This will allow us to define non-additive derived functors in Section~\ref{Section Derived Functors}.

\subsection{Subtractive functors}\label{sec - subtractive functors}
A first idea is to replace the condition that
\[
	F(f+g)=F(f)+F(g) \qquad\text{by}\qquad F(f-g)=F(f)-F(g).
\]
We will call such a functor \emph{subtractive}. Formally, this amounts to the following:

\begin{definition}[Subtractive functor]\label{def - subtractive functors}
	Let $F\colon \C\to \E$ be a functor between pointed categories with kernels where, in addition, $\C$ admits binary coproducts. We say that $F$ is a \defn{subtractive functor} if it preserves finite coproducts and $F(\delta_X)=\delta_{F(X)}$ for all $X\in \C$.
\end{definition}

\begin{remark}[Equivalent characterization of subtractive functors]\label{rem - subtractive functor via SSES}
	When both $\C$ and $\E$ are normal, this is equivalent to asking that binary coproducts and split short exact sequences of the form
	\[
		\xymatrix{0 \ar[r] & D(X) \ar@{{ |>}->}[r]^-{\delta_X} & X+X \ar@{-{ >>}}@<.5ex>[r]^-{\nabla_X} & X \ar[r] \ar@<.5ex>[l]^-{\iota_2} & 0}
	\]
	are preserved.
\end{remark}

This is not quite enough, though, for our purposes, because a problem occurs which is non-existent in the abelian case: we need preservation of proper morphisms, in order to guarantee that the derived functor will detect exact in the codomain of~$F$---something which is not implied by subtractivity. Combining these ingredients, we find:

\begin{proposition}\label{prop - homology VS functor - sub functor}
	Let $\C$ be a pointed category with kernels and binary coproducts. Let $\E$ be a normal subtractive category, and let $F\colon \C \to \E$ be a subtractive functor which preserves proper morphisms. If $f$, $g\colon (C,d^C)\to (E,d^E)$ are chain homotopic in $\C$, then $\H_n(F(f))=\H_n(F(g))$ for all $n$.
\end{proposition}
\begin{proof}
	As for additive functors in the additive context, any subtractive functor will preserve approximate chain homotopies: if $h\colon f\simeq g$, then $F(h)\colon F(f)\simeq F(g)$. The result follows from Theorem~\ref{thm - approx homotopy and homology} applied to the category $\E$.
\end{proof}

\subsection{Protoadditive functors preserving coproducts}
We now consider a convenient class of subtractive functors: the \emph{protoadditive functors} of T.~Everaert and M.~Gran~\cite{Everaert-Gran-TT,EG-honfg}.

\begin{definition}\label{def protoadditive functor}
	A functor $F\colon \C \to \E$ between pointed categories $\C$ and $\E$ is \defn{protoadditive} if it preserves kernels of split epimorphisms.
\end{definition}

\begin{remark}
	From the perspective of Definition~\ref{def protoadditive functor}, protoadditivity of a functor appears to be a weak kind of left exactness property. When $\C$ and $\E$ are normal categories, it is equivalent to the condition that $F$ preserves split short exact sequences: for any split short exact sequence
	\[
		\xymatrix{0 \ar[r] & K \ar@{{ |>}->}[r]^-k & X \ar@{-{ >>}}@<.5ex>[r]^-{f} & Y \ar[r] \ar@{{ >}->}@<.5ex>[l]^-{s} & 0}
	\]
	in $\C$, the image
	\[
		\xymatrix{0 \ar[r] & F(K) \ar@{{ |>}->}[r]^-{F(k)} & F(X) \ar@{-{ >>}}@<.5ex>[r]^-{F(f)} & F(Y) \ar[r] \ar@{{ >}->}@<.5ex>[l]^-{F(s)} & 0}
	\]
	by $F$ is a split short exact sequence in $\E$.

	Once $\C$ and $\E$ are abelian categories, any split short exact sequences is canonically induced by a biproduct. This implies that a functor between abelian categories is protoadditive if and only if it is additive---so that this left exactness aspect becomes invisible.
\end{remark}

The articles~\cite{Everaert-Gran-TT,EG-honfg} provide a list of examples: the reflector of commutative rings to reduced commutative rings, for instance, or any reflector from a category of topological semi-abelian algebras to Hausdorff algebras over the same theory. In Section~\ref{Sec:CrossedModules}, we focus on one in particular: the \emph{connected components} functor in the context of crossed modules.

\begin{lemma}
	Any binary coproduct-preserving protoadditive functor ${F\colon \C\to \E}$ between pointed categories with kernels $\C$ and $\E$, where $\C$ admits binary coproducts is a subtractive functor.\noproof
\end{lemma}

\begin{proposition}\label{prop - homology VS functor - protoadd with bin coproducts pres functor}
	Let $\C$ a pointed category with kernels and binary coproducts. Let~$\E$ be a normal subtractive category, and let $F\colon \C \to \E$ be a protoadditive functor which preserves binary coproducts and proper morphisms. If $f$, $g\colon (C,d^C)\to (E,d^E)$ are chain homotopic in $\C$ then $\H_n(F(f))=\H_n(F(g))$ for all~$n$. \noproof
\end{proposition}

\section{Derived functors}\label{Section Derived Functors}

We arrive at the main purpose of this article: to define derived functors in a non-additive context. In a pointed regular category $\C$ with binary coproducts and enough projectives which satisfies Condition \P, by Theorem~\ref{thm - projective resolution unique up to approx homoto}, any two projective resolutions of the same object $X\in \C$ are approximately chain homotopically equivalent. We let $\E$ be a homological category with binary coproducts, so that any two chain homotopic morphisms in it have the same homology. If we now consider a subtractive functor $F\colon \C\to \E$ which preserves proper morphisms as in Section~\ref{sec - discussion functors}, then for any $n\in \mathbb{Z}$ we may write
\[
	\Left_n(F)(X)\coloneq \H_n(F(C(X))).
\]
As explained in Section~\ref{sec - discussion functors}, the assumptions on the categories and on $F$ imply that we recover the essence of all three points valid in the abelian setting, mentioned in Theorem~\ref{Facts}. And indeed, together with the obvious extension to morphisms in~$\C$, the above procedure defines a functor $\Left_n(F)\colon \C\to \E$.

\begin{theorem}\label{Thm Def L}
	Let $\C$ be a pointed regular category with binary coproducts and with enough projectives which satisfies Condition \P. Let $\E$ be a normal subtractive category, and let $F\colon \C \to \E$ be a subtractive functor which preserves proper morphisms. Then $\Left_n(F)\colon \C\to \E$ is a functor for each $n\in \NN$.
\end{theorem}
\begin{proof}
	We already explained that up to isomorphism, this procedure is independent of chosen resolutions and liftings. It is clear that identities are preserved. Composites are preserved as well: if $f\colon X\to Y$, $g\colon Y\to Z$ and $gf\colon X\to Z$ lift to $C(f)\colon C(X)\to C(Y)$, $C(g)\colon C(Y)\to C(Z)$ and $C(gf)\colon C(X)\to C(Z)$, respectively, then both $C(g)\comp C(f)$ and $C(g\comp f)$ are liftings of $gf$, which makes them chain homotopic by Theorem~\ref{thm - projective resolution unique up to approx homoto}. Hence, $\Left_n(g)\Left_n(f)=\Left_n(gf)$ by Proposition~\ref{prop - homology VS functor - sub functor}.
\end{proof}

\begin{definition}\label{Def L}
	The functor $\Left_n(F)\colon \C\to \E$ is called the \defn{$n$th (non-additive) left derived functor} of $F$.
\end{definition}

Of course, in the abelian setting, this agrees with the classical definition. We always have:

\begin{proposition}\label{Prop Derived On Projective Object}
	Let $F$ be any functor as in Theorem~\ref{Thm Def L}. If $P$ is a projective object of $\C$, then $\Left_0(F)(P)=F(P)$ and $\Left_n(F)(P)=0$ for all $n\geq 1$.
\end{proposition}
\begin{proof}
	Remark~\ref{rem - projective resol when X proj} tells us that the chain complex which is $P$ in degree zero and~$0$ elsewhere is a projective resolution of $P$. The functor $F$ applied to it gives a chain complex whose homology is $F(P)$ in degree zero and $0$ elsewhere.
\end{proof}

Protoadditive functors often lead to protoadditive derived functors:

\begin{proposition}\label{Left derived of protoadditive functor}
	In the situation of Theorem~\ref{Thm Def L}, suppose moreover that $F$ is a protoadditive functor between homological categories. Then for each $n\in\NN$, the $n$th left derived functor $\Left_n(F)\colon \C\to \E$ is protoadditive as well.
\end{proposition}
\begin{proof}
	By Proposition~\ref{proposition:horseshoe lemma semiab} and Remark~\ref{Remark Horseshoe Lemma SSES}, any split short exact sequence
	\begin{equation}\label{SSESbis}
		\xymatrix{0 \ar[r] & K \ar@{{ |>}->}[r]^-k & X \ar@{-{ >>}}@<.5ex>[r]^-{f} & Y \ar[r] \ar@{{ >}->}@<.5ex>[l]^-{s} & 0}
	\end{equation}
	in the category $\C$ admits a split short exact sequence of resolutions
	\begin{equation*}
		\xymatrix{0 \ar[r] & A(K) \ar@{{ |>}->}[r] & C(X) \ar@<.5ex>@{-{ >>}}[r] & E(Y) \ar[r] \ar@{{ >}->}@<.5ex>[l] & 0\text{.}}
	\end{equation*}
	The functor $F$ applied to it yields the split short exact sequence of proper chain complexes
	\begin{equation*}
		\xymatrix{0 \ar[r] & F(A(K)) \ar@{{ |>}->}[r] & F(C(X)) \ar@<.5ex>@{-{ >>}}[r] & F(E(Y)) \ar[r] \ar@{{ >}->}@<.5ex>[l] & 0}
	\end{equation*}
	in the homological category $\E$. The induced long exact homology sequence~\cite[Theorem 4.5.7]{Borceux-Bourn} falls apart into the needed split short exact sequences of derived functors.
\end{proof}

\begin{remark}
	Note that we are assuming that the functor $F$ preserves binary coproducts. In the unlikely situation that a level-wise coproduct of any two projective resolutions in $\C$ is again a resolution (exactness is not guaranteed in general), the left derived functors $\Left_n(F)$ of $F$ preserve binary coproducts as well---which, as a consequence of Proposition~\ref{Left derived of protoadditive functor}, makes them subtractive functors.
\end{remark}

Recall~\cite{Bourn-Janelidze:Semidirect,BJK} that in a semi-abelian category, the middle object $X$ in a split short exact sequence \eqref{SSESbis} may be seen as a semidirect product $K\rtimes Y$ of the outer objects $K$ and $Y$, with respect to the induced internal action of $Y$ on $K$. Here the above result may be reformulated as follows:

\begin{corollary}\label{Protoadditive between semi-abelian cats}
	Let $\C$ be a semi-abelian category with enough projectives which satisfies Condition \P. Let $\E$ be semi-abelian and ${F\colon \C \to \E}$ a protoadditive functor which preserves binary coproducts and proper morphisms. Then
	\[
		\Left_n(F)(K\rtimes Y)\cong \Left_n(F)(K)\rtimes \Left_n(F)(Y)
	\]
	for each $n\in \NN$: the left derived functors of $F$ preserve semi-direct products.\noproof
\end{corollary}

When the target of $F$ is abelian, this becomes:

\begin{corollary}\label{Protoadditive with abelian target}
	Let $\C$ be a semi-abelian category with enough projectives which satisfies Condition \P. Let $\E$ be an abelian category and ${F\colon \C \to \E}$ a protoadditive functor which preserves binary coproducts. Then
	\[
		\Left_n(F)(K\rtimes Y)\cong \Left_n(F)(K)\oplus \Left_n(F)(Y)
	\]
	for each $n\in \NN$.
\end{corollary}
\begin{proof}
	This follows from the fact that in an abelian category, the middle object in any split short exact sequence is a biproduct of the outer objects.
\end{proof}

\subsection{Syzygy Lemma and long exact sequence in homology}\label{sec - syzygy}
We proceed by proving some basic results that hold for functors which are \emph{right exact} in an appropriate sense. Indeed, outside the abelian context, the concept of a right exact functor bifurcates. The following notion~\cite{PVdL1} is appropriate for our purposes:

\begin{definition}[Sequentially right exact functor]\label{Def Seq Right Exact Functor}
	A functor $F\colon \C\to \E$ where $\C$ and $\E$ are pointed categories is called \defn{sequentially right exact} when for each short exact sequence
	\begin{equation}\label{SES}
		\xymatrix{0\ar[r] & K\ar@{{ |>}->}[r]^-k & X \ar@{-{ >>}}[r]^-{f} & Y\ar[r] & 0}
	\end{equation}
	in $\C$, the sequence
	\[
		\xymatrix{F(K)\ar[r]^-{F(k)} & F(X) \ar@{-{ >>}}[r]^-{F(f)} & F(Y)\ar[r] & 0}
	\]
	is exact.
\end{definition}

\begin{example}\label{ex - sequentially right exact}
	If $\C$ is a semi-abelian category and $F\colon {\C\to \E}$ is the reflector from~$\C$ to a full replete regular epi--reflective subcategory $\E$ of $\C$, then $F$ is sequentially right exact. This is the situation considered in Section~\ref{Sec:CrossedModules} (where $F=\pi_0$).
\end{example}

This means, in particular, that the morphism $F(k)$ is proper. Note that in a category with kernels and cokernels, this is equivalent to the definition given in Subsection~\ref{Subsec Context}. In particular, sequentially right-exact functors preserve proper morphisms. Hence, for such a functor between categories $\C$ and $\E$ as in Theorem~\ref{Thm Def L}, it suffices that it is, in addition, subtractive for the theorem to be applicable. Furthermore, protoadditivity of $F$ follows as soon as it preserves protosplit monomorphisms. Examples of such functors include the connected components functors of Section~\ref{Sec:CrossedModules} (see \cite{EG-honfg, MC-TVdL-2} for further details).

We are now ready to prove versions of a few key fundamental results such as the existence of a long exact sequence in homology (see for instance~\cite[Theorem 2.4.6]{Weibel} for the abelian case) or the Syzygy Lemma~\cite[Exercise~2.4.3]{Weibel}.

\begin{lemma}\label{Lemma L_0}
	Let $\C$ be a homological category with binary coproducts and with enough projectives which satisfies Condition \P. Let $\E$ be a homological category with binary coproducts. If $F\colon \C\to \E$ is a sequentially right-exact functor that preserves protosplit monomorphisms and binary coproducts, then ${\Left_0(F)\cong F}$.
\end{lemma}
\begin{proof}
	Being a sequentially right-exact functor, $F$ preserves cokernels of proper morphisms. Hence $\Left_0(F)(X)=\H_0(F(C(X)))\cong F(X)$ for all $X$ in $\C$.
\end{proof}

\begin{theorem}[Long exact homology sequence]\label{them sec - long exact sequence of derived functors intro}
	Let $\C$ be a homological category with binary coproducts and with enough projectives which satisfies Condition \P. Let $\E$ be a homological category with binary coproducts. Let $F\colon \C\to \E$ be a sequentially right-exact functor that preserves protosplit monomorphisms and binary coproducts. Any short exact sequence \eqref{SES} in $\C$ gives rise to a long exact sequence in $\E$ as in Figure~\ref{fig long exact sequence with F} on page \pageref{fig long exact sequence with F}, which depends naturally on the given short exact sequence.
\end{theorem}
\begin{proof}
	The proof goes as in the abelian setting. By Proposition~\ref{proposition:horseshoe lemma semiab}, in the homological category $\C$, we have a resolution of the given short exact sequence with a split short exact sequence at each level. Since these are preserved by any protoadditive functor, we obtain a short sequence of proper chain complexes in $\E$. It suffices to apply~\cite[Theorem 4.5.7]{Borceux-Bourn} in the homological category $\E$ to find a long exact sequence as in Figure~\ref{fig long exact sequence with F}. The shape of the tail of the sequence (which \emph{a priori} consists of the functors $\Left_0(F)$) follows from Lemma~\ref{Lemma L_0}.
\end{proof}

A direct consequence of this result is the \defn{Syzygy Lemma}. Recall that an object $\Omega(X)$ is called a \defn{syzygy} of an object $X$ if it fits a short exact sequence of the shape
\begin{equation} \xymatrix{0\ar[r] & \Omega(X)\ar@{{ |>}->}[r] & P \ar@{-{ >>}}[r] & X\ar[r] & 0} \label{eq-def syzygy}
\end{equation}
where $P$ is a projective object. In a normal category with enough projectives, each object $X$ admits a syzygy $\Omega(X)$.

\begin{theorem}[Syzygy Lemma]\label{Syzygy Lemma}
	Under the assumptions of Theorem~\ref{them sec - long exact sequence of derived functors intro}, if $\Omega(X)$ is a syzygy of $X\in \C$, then
	\[
		\Left_{n+1}(F)(X)\cong \Left_n(F)(\Omega(X))
	\]
	for all $n\geq 1$.
\end{theorem}
\begin{proof}
	By Proposition~\ref{Prop Derived On Projective Object} we have $\Left_n(F)(P)=0$ for all $n> 0$. To conclude, we apply Theorem~\ref{them sec - long exact sequence of derived functors intro} on the short exact sequence \eqref{eq-def syzygy}.
\end{proof}

The case of Schreier varieties (such as the category of groups, for instance) merits some special attention, because in such a category, a syzygy is always a projective object.

\begin{corollary}\label{Cor Syzygy Schreier}
	Under the assumptions of Theorem~\ref{Syzygy Lemma}, if $\C$ is a Schreier variety, then $\Left_n(F)=0$ for all $n\geq 2$. \noproof
\end{corollary}

\section{Simplicial homotopy vs.\ approximate chain homotopy}\label{sec - comparison simplicial homotopy and approximate homotopy}

As explained in the introduction and in greater detail in Section~\ref{sec - approximate difference}, outside the additive context, we cannot define homotopy via~\eqref{Abelian-Homotopy}. This difficulty in defining a suitable notion of chain homotopy led the authors of~\cite{EverVdL2} to use simplicial methods---following~\cite{Barr-Beck,MR0470032,MR1489738}---rather than chain complexes in the context of semi-abelian homology. Barr and Beck~\cite{Barr-Beck} derived functors $F\colon \C \to \E$ where $\E$ is an abelian category, and Inassaridze~\cite{MR0470032,MR1489738} considered such $F$ where $\E$ is the category $\Gp$ of groups\footnote{The approach in~\cite{MacDonald}, however, is fundamentally different.}. In~\cite{EverVdL2}, both approaches are extended to the case where $\E$ is any semi-abelian category---although at the time, Inassaridze's work was not yet known to the authors of~\cite{EverVdL2}. Aside from the ad-hoc strategy of fixing a functorial choice of resolutions once and for all, this was the only general approach to homology readily available in this context. The existence of standard simplicial tools in the semi-abelian context naturally led the authors of~\cite{JuliaThesis} and~\cite{GVdL} to use simplicial methods for proving results in homological algebra.

Our aim in this section is to investigate under which conditions the simplicially defined derived functors of~\cite{EverVdL2} agree with those of Section~\ref{Section Derived Functors}. We start by recalling some basic definitions.

\subsection{Simplicial objects and homology}
When $\C$ is an abelian category, we have the well-known \defn{Dold--Kan--Puppe equivalence}~\cite{Dold-Puppe,DoldcorrespondenceRMod}, which asserts that the category of simplicial objects in $\C$ (denoted $\s(\C)$) and the category of positively graded chain complexes in $\C$ (denoted $\Ch^+(\C)$) are equivalent. In a semi-abelian category, this need no longer be the case. For instance, in~\cite{Carrasco-Homotopy}, Carrasco and Cegarra provide an explicit description of $\s(\Gp)$ as being equivalent to the category of \emph{hypercrossed complexes of groups}, which differs from $\Ch^+(\Gp)$. This fact was generalized by Bourn~\cite{Bourn:Dold-Kanpublished} to semi-abelian categories.

We can, however, still associate a chain complex to any simplicial object, even in a non-abelian category. Indeed, the \defn{Moore functor} $\N\colon \s(\C)\to \Ch^+(\C)$ can be defined for any pointed category with pullbacks $\C$. It has been shown in~\cite{EverVdL2,Borceux-Bourn} that if $\C$ is, in addition, exact and protomodular, then $\N$ maps any simplicial object to a proper chain complex---hence the concept of \defn{homology of a simplicial object}, defined as the homology of the associated Moore chain complex.

For the sake of convenience, we recall some of the details. A \defn{semi-simplicial object} $\A=(A_n)_{n\geq 0}$ in a category $\C$ is a family of objects, together with \defn{face operators} $\partial_{i}^\A\colon A_n\to A_{n-1}$ for $i\in [n]=\{0,\dots, n\}$ where $n>0$, such that
\begin{equation*}\label{eq - semi-simplicial identities}
	\partial_{i}^\A \partial_{j}^\A  = \partial_{{j-1}}^\A \partial_{i}^\A
\end{equation*}
holds if $i<j$. A \defn{quasi-simplicial object} is a semi-simplicial object equipped with \defn{degeneracy operators} $\sigma_{i }^\A\colon A_n \to A_{n+1}$ for $n\geq 0$, such that
\begin{equation*}\label{eq - quasi-simplicial identities}
	\partial_{i}^\A \sigma_{j}^\A    = \begin{cases}
		\sigma_{{j-1}}^\A \partial_{i}^\A & \text{if  $i<j$}            \\
		1_{A_n}                           & \text{if  $i=j$ or $i=j+1$} \\
		\sigma_{j}^\A \partial_{{i-1}}^\A & \text{if  $i>j+1$}
	\end{cases}
\end{equation*}
Finally, a quasi-simplicial object is called a \defn{simplicial object} if the degeneracy operators satisfy
\begin{equation*}\label{eq - simplicial identities}
	\sigma_{i }^\A \sigma_{j }^\A    = \sigma_{{j+1}}^\A \sigma_{{i}}^\A
\end{equation*}
when $i<j$.

A \defn{(quasi-)simplicial morphism} from  $\A$ to $\B$ is a family ${f=(f_n\colon A_n\to B_n)_{n\geq 0}}$ of morphisms of $\C$, such that
\[
	f_{n-1}\partial_{i}^\A=\partial_{i}^\B f_n \qquad\text{and}\qquad f_{n +1}\sigma_i^\A = \sigma_i^\B f_n.
\]

Let $\A$ be a (semi-)simplicial object in a pointed category with pullbacks $\C$. The \defn{Moore chain complex} or \defn{normalized chain complex} $\N(\A)$ is the chain complex defined by $\N_{n}(\A)=0$ for $n<0$, $\N_0(\A)=A_0$ and
\[
	\N_{n}(\A)\coloneq \Ker((\partial_{l}^\A)_{0\leq l \leq n-1}\colon A_n \to A_{n-1}^n)=\bigcap_{l=0}^{n-1}\Ker(\partial_{l}^\A\colon A_n \to A_{n-1})
\]
when $n\geq 1$. We denote the canonical inclusion $\kappa_n^\A\colon \N_n(\A)\to A_n$. The boundary morphism $d^A_n\colon \N_{n}(\A)\to \N_{n-1}(\A)$ is the restriction of $\partial_{n}^\A\colon A_n\to A_{n-1}$ to the appropriate kernel, and zero when $n\leq 0$. Clearly, $(\N(\A),d^A)$ is a chain complex, so the above definition gives rise to a functor $\N\colon \s(\C)\to \Ch^+(\C)$ where for a given simplicial morphism $f=(f_n)_n$, the chain morphism $\N_n(f)$ is defined by restricting $f_n$ to the kernels $\N_n(\A)$ and $\N_n(\B)$.

If in addition the category $\C$ is homological, then the functor $\N$ will preserve and reflect exactness, i.e., the image of an exact quasi-simplicial object by $\N$ is an exact chain complex and vice versa~\cite[Theorem 3.9]{Goedecke}.

\subsection{Simplicial homotopy}
Given simplicial morphisms $f$, $g\colon \A \to \B$, a \defn{(simplicial) homotopy} $\h$ from $f$ to $g$ consists of a family of morphisms $(\h_i\colon A_n\to B_{n+1})_{0\leq i\leq n}$ in $\C$ satisfying
\begin{align}\partial_{0}^\B \h_0  =f_n \qquad                                                     & \text{and}\qquad \partial_{{n+1}}^\B \h_n=g_n, \notag \\
             \partial_{i}^\B \h_j  = \begin{cases}
	                                     \h_{j-1} \partial_{i }^\A & \text{if $i<j$}       \\
	                                     \partial_{i}^\B \h_{i-1}  & \text{if $i=j\neq 0$} \\
	                                     \h_j \partial_{{i-1}}^\A  & \text{if $i>j+1$}
                                     \end{cases}   \quad & \text{and}\quad
             \sigma_{i}^\B \h_j    = \begin{cases}
	                                     \h_{j+1}\sigma^\A_i    & \text{if $i\leq j$} \\
	                                     \h_j \sigma_{{i-1}}^\A & \text{if $i>j$}
                                     \end{cases}
             \label{def - simplicial homotopy}
\end{align}
for each $n\geq 0$. We write $\h\colon f\simeq g$.

In the abelian setting, the Dold--Kan--Puppe equivalence extends to homotopies: through the equivalence, homotopic simplicial morphisms correspond to homotopic chain morphisms (see for instance~\cite[Theorem 8.4.1]{Weibel}). One implication remains true in semi-abelian categories: if $\h\colon f\simeq g\colon \A \to \B$, then an approximate chain homotopy $\N(f)\simeq\N(g)$ exists (Theorem~\ref{thm - simplicial homotopy implies chain homotopy - nonadd case}). In particular, then $\H_n\N(f)=\H_n\N(g)$, so that we recover a result of~\cite{JuliaThesis,GVdL}.

Given a simplicial homotopy $\h\colon f\simeq g\colon \A \to \B$, our aim is now to construct an approximate chain homotopy $h\colon\N(f)\simeq\N(g)$. We first observe that from the simplicial identities and the definition of a simplicial homotopy, we may deduce
\begin{equation*}
	\partial_{i }^\B(\sigma_{j}^\B f_n-\h_j)= \begin{cases}
		(\sigma_{{j-1}}^\B f_{n-1} - \h_{j-1})D(\partial_{i}^\A) & \text{if $i<j$}            \\
		f_n-\partial_{i}^\B \h_i=f_n-\partial_{i}^\B \h_{i-1}    & \text{if $i=j$ or $i=j+1$} \\
		(\sigma_{j}^\B f_{n-1}-\h_j)D(\partial_{{i-1}}^\A)       & \text{if $i>j+1$}
	\end{cases}
\end{equation*}
for any $0\leq i <n-1$, so that
\begin{equation}
	\partial_{i }^\B(\sigma_{j}^\B f_n-\h_j)D(\kappa_n^\A)=0
	\label{equation for homotopy - nonab case}
\end{equation}
when $i<j$ or $i>j+1$.
Furthermore, by Lemma~\ref{lemma: distributivity for minus}, we have
\[
	\partial_{i }^\B(\sigma_{j}^\B f_n-\h_j)D(\kappa_n^\A ) =\partial_{i }^\B(\sigma_{j}^\B f_n \kappa_n^\A -\h_j \kappa_n^\A ).
\]

\subsubsection*{Construction of $h_0$}
By the definition of the simplicial homotopy $\h$, the morphism $\sigma^B_0f_0-\h_0\colon D(\N_0(\A)) \to B_1$ satisfies
\[
	\partial_{0}^\B(\sigma_{0}^\B f_0-\h_0)= f_0-\partial_{0}^\B\h_0= f_0-f_0=0.
\]
Hence a unique morphism $h_0\colon D(\N_0(\A))\to \N_1(\B)$ exists (see Figure~\ref{Diag h_0}) such that
\[
	\kker(\partial_{0}^\B)h_0 = \sigma_{0}^\B f_0 - \h_0.
\]
Now
\[
	d_1^B h_0 =\partial_{1}^\B \kker(\partial_{0}^\B) h_0 = \partial_{1}^\B(\sigma_{0}^\B f_0-\h_0)=f_0-g_0.
\]

\begin{figure}
	$
		\xymatrix@R=4em@C=6em{
		A_1\ar[r]^-{\partial_1^A} &  A_0&\\
		\N_1(\A)\ar@{{ |>}->}[u]^-{\kappa_1^{\A}=\kker(\partial_0^A)} \ar@<-1 ex>[d]_-{\N_1(f)}\ar@<.5ex>[d]^-{\N_1(g)}\ar[r]^-{d_1^A} &  \N_0(\A)\ar@{=}[u]\ar@<-1 ex>[d]_-(0.3){f_0}\ar@<.5ex>[d]^-(0.3){g_0}& D(\N_0(\A))\ar@/^2pc/@{.>}[ddll]^-(0.3){\sigma^B_0f_0-\h_0}\ar@/^0.5pc/@{-->}[dll]_-(0.25){h_0}\\
		\N_1(\B)\ar@{{ |>}->}[d]_-{\kappa_1^{\B}=\kker(\partial_0 ^B)} \ar[r]_{d_1^B}  & \N_0(\B)\ar@{=}[d]\\
		B_1\ar[r]_-{\partial_1^B} & B_0&
		}
	$
	\caption{Diagram for $h_0$}\label{Diag h_0}
\end{figure}

\subsubsection*{Construction of $h_1$}
It suffices to consider the morphism
\begin{multline*}
	\left( (\sigma_{1}^\B f_1-\h_1) - (\sigma_{0 }^\B f_1 - \h_0)\right)D^2(\kker(\partial_{0}^\A))         \\
	= (\sigma_{1}^\B f_1-\h_1)D(\kker(\partial_{0}^\A)) - (\sigma_{0 }^\B f_1 - \h_0)D(\kker(\partial_{0}^\A))
\end{multline*}
which, thanks to~\eqref{def - simplicial homotopy}, \eqref{equation for homotopy - nonab case} and naturality of $\varsigma$, is such that
\begin{align*}
	 & \partial_{0}^\B \left( (\sigma_{1}^\B f_1-\h_1) - (\sigma_{0 }^\B f_1 - \h_0)\right)D^2(\kker(\partial_{0}^\A)) \\
	 & = ((\sigma_{0}^\B f_0-\h_0)D(\partial_{0}^\A)-0) D^2(\kker(\partial_{0}^\A))                                    \\
	 & = (\sigma_{0}^\B f_0-\h_0)D(\partial_{{0}}^\A)\varsigma_{D(A_1)} D^2(\kker(\partial_{0}^\A))                    \\
	 & = (\sigma_{0}^\B f_0-\h_0)D(\partial_{0}^\A) D(\kker(\partial_{0}^\A)) \varsigma_{D(\N_1(\A))} =0
\end{align*}
and
\begin{align*}
	 & \partial_{1}^\B \left( (\sigma_{1}^\B f_1-\h_1) - (\sigma_{0 }^\B f_1 - \h_0)\right)D^2(\kker(\partial_{0}^\A)) \\ & = ((f_1-\partial_{1}^\B \h_1)-(f_1-\partial_{1}^\B\h_0))D^2(\kker(\partial_{0}^\A)) =0.
\end{align*}
Consequently, there exists a unique morphism $h_1\colon D^2(\N_1(\A))\to \N_2(\B)$ such that
\[
	\kappa_2^\B h_1 = \left( (\sigma_{1}^\B f_1-\h_1) - (\sigma_{0 }^\B f_1 - \h_0)\right)D^2(\kker(\partial_{0}^\A)).
\]
By Lemma~\ref{lemma: distributivity for minus},
\[
	\kker(\partial_{0}^\B) d_2^B h_1 = \kker(\partial_{0}^\B) (\N_1(f)-\N_1(g))- \kker(\partial_{0}^\B) h_0 D(d_1^B),
\]
and since $\kker(\partial_{0}^\B)$ is monic, this implies
\[
	d_2^B h_1 =(\N_1(f)-\N_1(g))-  h_0 D(d_1^B).
\]

For $n\geq 2$, the morphisms that we must consider are slightly more complicated, since they will involve the natural transformation $\varsigma$.

\subsubsection*{Construction of $h_2$} When $n=2$, we put
\[
	\tilde{h}_2\coloneq \Bigl((\sigma_{2}^\B f_2 - \h_2) \varsigma_{D(A_2)} -\bigl((\sigma_{1}^\B f_2 - \h_1)- (\sigma_{0}^\B f_2 - \h_0) \bigr)\Bigr) D^3(\kappa_2^\A).
\]
Then there exists a unique morphism $h_2$ such that $\kappa_3^\B h_2 = \tilde{h}_2$. We find
\[
	d^B_3 h_2 = (\N_2(f)-\N_2(g))\varsigma_{D(\N_2(\A))}-h_1D^2(d_2^A).
\]

\subsubsection*{Construction of $h_n$} We can proceed similarly for defining $h_n$ when $n\geq 2$:
\begin{align*}
	\tilde{h}_n  \coloneq & \;\Bigl((\sigma_{n}^\B f_n- \h_n)\varsigma_{D(A_n)}^{n-1} -\bigl((\sigma_{{n-1}}^\B f_n- \h_{n-1})\varsigma_{D(A_n)}^{n-2}                                  \\
	                      & -(\cdots -((\sigma_{2}^\B f_n- \h_2)\varsigma_{D(A_n)}-((\sigma_{1}^\B f_n - \h_1)-(\sigma_{0}^\B f_n - \h_0))) \cdots )\bigr)\Bigr)D^{n+1}(\kappa_n^\A)    \\
	=                     & \; (\sigma_{n}^\B f_n- \h_n)\varsigma_{D(A_n)}^{n-1} D^{n}(\kappa_n^\A)-\bigl((\sigma_{{n-1}}^\B f_n- \h_{n-1})\varsigma_{D(A_n)}^{n-2}D^{n-1}(\kappa_n^\A) \\
	                      & -(\cdots-((\sigma_{2}^\B f_n- \h_2)\varsigma_{D(A_n)}D^{2}(\kappa_n^\A)                                                                                     \\
	                      & \qquad\;\;\,-((\sigma_{1}^\B f_n - \h_1)D(\kappa_n^\A)-(\sigma_{0}^\B f_n - \h_0)D(\kappa_n^\A)) \cdots ))\bigr)
\end{align*}
Using again the Equations~\eqref{equation for homotopy - nonab case} and \eqref{eq - compatibility sigma's and D's}, we can prove that $\partial_{i}^\B \tilde{h}_n=0$ for all $0\leq i \leq n$. Therefore, there exists a unique $h_n$ such that $\kappa_{n+1}^\B h_n = \tilde{h}_n$ as in Figure~\ref{Diag h_n}. Finally, by definition of $h_n$ and of $h_{n-1}$ for $n\geq 2$ and by using again the Equation~\eqref{equation for homotopy - nonab case} and Lemma~\ref{lemma: distributivity for minus}, we find:
\[
	d^B_{n+1} h_n = (\N_n(f)-\N_n(g))\varsigma_{D(\N_n(\A))}^{n-1}-h_{n-1}D^n(d_n^A).
\]

\begin{figure}
	$
		\xymatrix@R=4em@C=6em{
		A_{n+1}\ar[r]^-{\partial_{n+1}^\A} &  A_n&\\
		\N_{n+1}(\A)\ar@{{ |>}->}[u]^-{\kappa_{n+1}^{\A}} \ar@<-1 ex>[d]_-{\N_{n+1}(f)}\ar@<.5ex>[d]^-{\N_{n+1}(g)}\ar[r]^-{d_{n+1}^A} &  \N_n(\A)\ar@<-1 ex>[d]_-(0.3){\N_{n}(f)}\ar@<.5ex>[d]^-(0.3){\N_{n}(g)}  \ar@{{ |>}->}[u]_-{\kappa_n^{\A}}&D^{n+1}(\N_n(\A))\ar@/^2pc/@{.>}[ddll]^-(0.3){\tilde{h}_n}\ar@/^0.5pc/@{-->}[dll]_-(0.25){h_n}\\
		\N_{n+1}(\B)\ar@{{ |>}->}[d]_-{\kappa_{n+1}^\B}  \ar[r]_-{d_{n+1}^B}  & \N_n(\B)\ar@{{ |>}->}[d]^-(0.3){\kappa_n^{\B}}&\\
		B_{n+1}\ar[r]_-{\partial_{n+1}^\B} & B_{n}&
		}
	$
	\caption{Diagram for $h_n$}\label{Diag h_n}
\end{figure}

This proves the next result.
\begin{theorem}\label{thm - simplicial homotopy implies chain homotopy - nonadd case}
	Let $f$, $g\colon \A \to \B$ be simplicial morphisms in a semi-abelian category. If $f$ and $g$ are simplicially homotopic, then the chain morphisms $\N(f)$ and $\N(g)$ are approximately chain homotopic (Definition~\ref{Def Approx Chain Homotopy}).\noproof
\end{theorem}

By Theorem~\ref{thm - approx homotopy and homology}, we thus recover:

\begin{corollary}
	Let $f$, $g\colon \A \to \B$ be simplicial morphisms in a semi-abelian category. If $f\simeq g$ then we have $\H_n\N(f)=\H_n\N(g)$ for all $n\geq 0$. \noproof
\end{corollary}

\subsection{Simplicial resolutions vs.\ chain resolutions}
Simplicial objects can be used to define \emph{simplicially derived functors}. Our aim here is to compare these with the chain derived functors of Section~\ref{Section Derived Functors} in contexts where both are available.

Recall that an \defn{augmented simplicial object} $(\S,\del_0^\S)$ is a simplicial object $\S$ together with a morphism $\del_0^\S\colon S_0\to S_{-1}$ such that $\del_0^\S\del_0^\S=\del_0^\S\del_1^\S$. In the context of a homological category\footnote{The concept of a (semi-)simplicial resolution goes back at least to~\cite{Dold-Puppe, Tierney-Vogel2, Tierney-Vogel, Barr-Beck}
	and makes sense in contexts, far more general than homological categories. The current setting allows a simpler presentation by means of an exact chain complex~\cite[Theorem~4.11]{Goedecke}, while sufficing for our purposes.}, an augmented simplicial object $(\S,\del_0^\S)$ is a \defn{(projective) simplicial resolution (of the object $S_{-1}$)} when the sequence
\begin{equation}\label{Eq Normalization of Resolution}
	\xymatrix{\cdots \ar[r] & \N_{n+1}(\S) \ar[r]^-{d_{n+1}^S} & \N_{n}(\S) \ar[r] & \cdots \ar[r]^-{d_1^S} & \N_0(\S) \ar[r]^-{\del_0^\S} & S_{-1} \ar[r] & 0}
\end{equation}
is exact (so that, in particular, $S_{-1}=\H_0(\N(\S))$) and $n\geq 0$ implies that $S_n$ is projective.

Barr and Beck~\cite{Barr-Beck} use comonads to construct resolutions (which are potentially not of the above kind, depending on the chosen comonad), while Tierney and Vogel~\cite{Tierney-Vogel2} use a recursive construction (for something more general involving a semi- or quasi-simplicial object). Here, it suffices to recall that in each case, any two given resolutions are homotopically equivalent, so that they can be used to define derived functors. For instance, if $\C$ is a variety of algebras and $\G$ is the comonad induced by the forgetful adjunction to $\Set$, then each object $X$ of $\C$ admits a canonical simplicial resolution we denote $(\G(X),\del_0^{\G(X)})$. Here the definition recalled above does indeed agree with Barr and Beck's~\cite[Theorem~4.11]{Goedecke}. For any functor $F\colon \C\to \E$ where $\E$ is a semi-abelian category and for each $n\geq 0$, we may define the \defn{$(n+1)$th simplicially derived functor $\H_{n+1}(-,F)_{\G}$ of $F$} via
\[
	\H_{n+1}(X,F)_{\G}\coloneq\H_n(\N(F(\G(X)))),
\]
for $X$ in $\C$. Note the dimension shift which is there for historical reasons; see, for instance, \cite{EverVdL2, JuliaThesis} for further details.

On the way to proving that under certain conditions, these derived functors may be obtained via projective chain resolutions, we are now going to explore when the normalization $\N(\S)$ of a simplicial resolution $(\S^-,\del_0^{\S^-})$ of an object $X$ is a projective chain resolution of $X$. As explained above, in a homological category, we always find an exact sequence \eqref{Eq Normalization of Resolution} of the right shape. We just need that all the objects $\N_i(\S)$ are projective.

This depends on a dimension shift argument, based on the concept of \emph{décalage} of an augmented simplicial object. Given such an $(\S,\del_0^\S)$ in a homological category~$\C$, its \defn{décalage} is the augmented simplicial object $(\S^-,\del_0^{\S^-})$ obtained by leaving out the $S_{-1}$, all the $\partial^{\S}_0\colon S_n\to S_{n-1}$ and all the $\sigma^{\S}_0$, so that $S^-_{n-1}=S_{n}$, $\del^{\S^-}_i=\del^{\S}_{i+1}\colon S_{n+1}\to S_n$, etc. Those left out morphisms conspire to a split epimorphism of simplicial objects, and thus a short exact sequence
\begin{equation}\label{eq - shifting and Lambda simplicial object}
	\xymatrix{0 \ar[r] & \Lambda(\S) \ar@{{ |>}->}[r] & {\S}^- \ar@{-{ >>}}@<.5ex>[r]^-{(\partial^{\S}_0)_n} & \S \ar[r] \ar@{{ >}->}@<.5ex>[l]^-{(\sigma^{\S}_0)_n} & 0}
\end{equation}
in $\s(\C)$. Since the functor $\N$ is exact (see~\cite[Theorem 3.9]{Goedecke}), we obtain a split short exact sequence of chain complexes
\begin{equation*}\label{Eq N on shift}
	\xymatrix{0 \ar[r] & \N(\Lambda(\S)) \ar@{{ |>}->}[r] & \N(\S^-) \ar@{-{ >>}}@<.5ex>[r] & \N(\S) \ar[r] \ar@{{ >}->}@<.5ex>[l] & 0.}
\end{equation*}
By \cite[Lemma 3.6]{Goedecke}, we have $\H_{n}\N(\Lambda(\S))=\H_{n+1}\N(\S)$ for all $n\geq 1$, essentially by the long exact homology sequence induced by this short exact sequence of chain complexes and the fact that $\H_n\N(\S^-)=0$ when $n\geq 1$. However, the following stronger result holds as well, and does in fact follow immediately from the definition of the Moore complex and \eqref{eq - shifting and Lambda simplicial object}: for all $n\geq 1$, we have $\N_{n-1}\Lambda(\S)=\N_{n}\S$ and $d_{n}^{\Lambda S}=d_{n+1}^S$. From this, it is easy to deduce that $\N_n(\S)=(\Lambda^n (\S))_0$ for all~${n\geq 1}$.

If now Condition \P\ holds in $\C$, then if a simplicial object $\S$ in $\C$ consists of projective objects, so does $\S^-$, and hence $\Lambda (\S)$ as well, by the exactness of~\eqref{eq - shifting and Lambda simplicial object}. By induction, it follows that all objects of $\N(\S)$ are projective.

In conclusion: the normalization of a projective simplicial resolution $\S(X)$ of an object $X$ is a projective chain resolution of $X$.

\begin{proposition}\label{thm - NS=P}
	Let $\C$ be a homological category with enough projectives that satisfies Condition \P, and let $X$ be an object of $\C$. For any simplicial resolution $\S(X)$ of $X$, the Moore complex $\N(\S(X))$ is a projective chain resolution of $X$.\noproof
\end{proposition}

Since the projective chain resolution $\N(\S(X))$ will be approximately chain homotopically equivalent to any projective chain resolution $C(X)$ of $X$, it is useful to compare the results of applying a functor to both $\S(X)$ and $C(X)$. If the functor $F$ is protoadditive, we can apply it level-wise to~\eqref{eq - shifting and Lambda simplicial object} and observe that $F(\Lambda(\S(X)))=\Lambda(F(\S(X)))$. Consequently,
\[
	F(\N_n(\S(X)))=F((\Lambda^n(\S(X)))_0)=\Lambda^n(F(\S(X)))_0=\N_n(F(\S(X)))
\]
for all $n\geq 1$, which implies
\[
	\H_n(\N(F(\S(X))))=\H_n(F(\N(\S(X))))
\]
for all $n\geq 0$. Therefore:
\begin{theorem}\label{Thm Compa}
	Let $\C$ be a homological category with binary coproducts and enough projectives that satisfies Condition \P, let $\E$ be a semi-abelian category, and let $F\colon \C\to \E$ be a protoadditive functor that preserves binary coproducts and proper morphisms. Then, for any object $X$, any choice of simplicial resolution $\S(X)$, and any chain resolution $C(X)$, we have
	\[
		\H_n(\N(F(\S(X))))=\H_n(F(\N(\S(X))))=\H_n(F(C(X))).
	\]

	If moreover $\C$ is a variety of algebras and $\G$ is the comonad induced by the forgetful adjunction to $\Set$, then for each $n\geq 0$, the derived functor $\Left_n(F)$ is naturally isomorphic to $\H_{n+1}(-,F)_{\G}$.
	\noproof
\end{theorem}

In the next section, we consider an example of this situation in the context of crossed modules.

\section{Example: crossed modules and \texorpdfstring{\ncat-groups}{ncat-groups}}\label{Sec:CrossedModules}

The category $\XMod$ of crossed modules~\cite{Wh,MW}, being a variety of $\Omega$-groups~\cite{Higgins}, is a semi-abelian category with enough projectives~\cite{Janelidze-Marki-Tholen}. More generally, such is the variety $\nCatGrp$ of \ncat-groups~\cite{Loday} for each $n\geq 1$. It is known that $\XMod$ is not a Schreier variety: in \cite[Proposition 5]{Carrasco-Homology}, a projective object is given that is not free. On the other hand, in~\cite{MC-TVdL-2} we find:

\begin{theorem}\label{Thm Crossed Modules}
	In the category $\XMod$ of crossed modules, the class of projectives is closed under protosplit subobjects.\noproof
\end{theorem}

It is further shown there that more generally, the non-Schreier variety $\nCatGrp$ of \ncat-groups satisfies Condition \P\ for each $n\geq 1$. This is a consequence of the fact~\cite{MC-TVdL-2} that the category of internal crossed modules~\cite{Janelidze} in any semi-abelian variety of universal algebras whose class of projectives is closed under protosplit subobjects again satisfies Condition \P. For each $n\geq 1$, the category $\nCatGrp$ is equivalent to the category $\Cat^n(\Gp)$ of $n$-fold categories in $\Gp$, and the result follows by induction from the equivalence between internal crossed modules and internal categories.

A concrete example of a protoadditive functor in the context of crossed modules is the \defn{connected components functor} $\pi_0\colon \XMod\to \Gp$. Let us recall~\cite{BrownSpencer,MacLane} that any crossed module $(T,G,\phi)$ can be viewed as follows as an internal category in $\Gp$. Its underlying reflexive graph occurs in the diagram
\[
	\xymatrix{0 \ar[r] & T \ar@{{ |>}->}[r]^-k & T\rtimes G \ar@{-{ >>}}@<1ex>[r]^-{d}\ar@<-1ex>[r]_-{c} & G \ar[r] \ar@{{ >}->}[l]|-{e} & 0}
\]
where $k=\ker(d)$, $ce=1_G$ and $ck=\phi$, and the semidirect product is taken with respect to the given action of $G$ on $T$. Hence, $\pi_0\left((T,G,\phi) \right)$ may be defined as the coequalizer of $d$ and $c$, or equivalently~\cite[Proposition 3.9]{EverVdL2} as the cokernel of $ck=\phi$. The functor $\pi_0$ is left adjoint to the \defn{discrete crossed module functor} ${D\colon \Gp \to \XMod}$, which associates to a group $G$ the (discrete) crossed module ${0\to G}$. This implies, in particular, that $\pi_0$ preserves all coproducts. Moreover, $\Gp$ can be viewed as the subvariety of $\XMod$ consisting of the discrete crossed modules. From this last observation, it follows that $\pi_0$ preserves proper morphisms: actually, it is sequentially right exact (see Example~\ref{ex - sequentially right exact}). Finally, as shown in~\cite{EG-honfg}, the functor is protoadditive.

By Theorem~\ref{Thm Def L}, the non-additive derived functors of the protoadditive functor $\pi_0\colon \XMod \to \Gp$ are well defined. By Theorem~\ref{Thm Compa}, these are equivalent to the simplicially derived functors (in the sense of~\cite{Barr-Beck}) of the functor $\pi_0$, which are known to be expressible in terms of Hopf formulae~\cite{EGVdL,EG-honfg}. This provides a new perspective on these Hopf formulae, and more generally on the Hopf formulae of \cite[Theorem~3.3]{EG-honfg}, which may now be obtained as the chain derived functors of the natural composite of connected components functors
\[
	\nCatGrp\to \CatGrp{(n-1)}\to\dots\to \CatGrp{1}\simeq\XMod \to \Gp.
\]
Something very similar happens in~\cite{Donadze-VdL}. Further details, including detailed proofs and a general result for internal crossed modules in any semi-abelian category that satisfies Condition \P, are given in the article~\cite{MC-TVdL-2}.

\section{Further questions}\label{sec:questions}

The condition that the class of projective objects is closed under protosplit subobjects is quite new and not yet very well understood. We have seen that it holds for Lie algebras over a commutative unitary ring, for of crossed modules, and more generally for \ncat-groups, besides all Schreier varieties of algebras and all abelian categories. It is not clear how common this condition is, and what its implications are. We list here some further problems that we believe are worth investigating.

\begin{enumerate}
	\item Characterize those varieties of algebras that satisfy Condition \P.
	\item For a protoadditive functor whose domain satisfies Condition \P, simplicially derived functors are equivalent to our chain derived functors.

	      Protoadditive functors between semi-abelian categories are relatively rare compared to additive functors between abelian categories. (For instance, there are no non-trivial protoadditive reflections of the category of groups~\cite{CGJ}.) Find new examples and characterizations to better understand them.
	\item In general, explore the relationship between the chain derived functors of a functor between semi-abelian varieties of algebras and the simplicially derived functors of the same functor.
	\item Explore approximate chain homotopy from a homotopy-theoretical perspective. In forthcoming work, we already checked that it can be characterized by means of a \emph{cylinder object}, but we did not yet examine the homotopy theory of the category of chain complexes with respect to this notion (i.e., quotient out the homotopy relation) from a general, abstract perspective.

	      Natural questions in this context include, whether there exists a Quillen model structure~\cite{Quillen} (or perhaps a weaker, one-sided version of this notion such as the one considered in~\cite{K.S.Brown}) on the category of chain complexes with respect to approximate chain homotopy, and whether the derived functors can be computed by means of this model structure.
	\item We did not investigate how to deal with \emph{right} derived functors or with duality in general. Neither did we study those functors which typically lead to cohomology groups, as in~\cite{RVdL2,PVdL1} for instance. Here the question arises, for which kind of abelian object $A$ of a semi-abelian category $\C$, the functor $\Hom(-,A)\colon \C^{\op}\to \Ab$ satisfies the conditions considered in this article.
\end{enumerate}

\section*{Acknowledgments}
Many thanks to Jacques Darn\'e for his insightful remarks and suggestions concerning the Lazard decomposition theorem, and to the referee for careful comments and suggestions leading to the present, improved version of the text.

%\bibliography{tim}
%%\bibliographystyle{amsplain}
%\bibliographystyle{amsplain-nodash}

\end{document}